\documentclass{amsart}

\usepackage{amsmath,amsthm, amssymb,mathrsfs,  xypic}
\usepackage{stackengine,scalerel}
\usepackage[colorlinks]{hyperref}

 \newcommand{\Nilp}{\mathrm{Nilp}}

\renewcommand{\AA}{\mathbb{A}}

\renewcommand\widetilde[1]{\ThisStyle{%
  \setbox0=\hbox{$\SavedStyle#1$}%
  \stackengine{-.1\LMpt}{$\SavedStyle#1$}{%
    \stretchto{\scaleto{\SavedStyle\mkern.2mu\sim}{.5467\wd0}}{.7\ht0}%
  }{O}{c}{F}{T}{S}%
}}

\newcommand{\QQ}{\mathbb{Q}}
\newcommand{\ZZ}{\mathbb{Z}}

\newcommand{\GG}{\mathbb{G}}

\newcommand{\VV}{\mathbb{V}}

\renewcommand{\det}{\mathrm{det}}

\newcommand{\tensorhat}{\widehat{\otimes}}

\newcommand{\isoeq}{\cong}

\newcommand{\Id}{\mathrm{Id}}
\newcommand{\std}{\mathrm{std}}
\newcommand{\dR}{\mathrm{dR}}

\newcommand{\Aut}{\mathrm{Aut}}

\newcommand{\Spf}{\mathrm{Spf}}
\newcommand{\Spec}{\mathrm{Spec}}
\newcommand{\Hom}{\mathrm{Hom}}

\newcommand{\Lie}{\mathrm{Lie}}

\newcommand{\Ext}{\mathrm{Ext}}

\newcommand{\GL}{\mathrm{GL}}

\newcommand{\et}{\mathrm{\acute{e}t}}

\newcommand{\Cont}{\mathrm{Cont}}

\newcommand{\Katz}{\mathrm{Katz}}

\newcommand{\Ig}{\mathrm{Ig}}

\newcommand{\univ}{\mathrm{univ}}

\newcommand{\ur}{\mathrm{ur}}

\newcommand{\can}{\mathrm{can}}

\newcommand{\calD}{\mathcal{D}}
\newcommand{\calE}{\mathcal{E}}
\newcommand{\calF}{\mathcal{F}}
\newcommand{\calL}{\mathcal{L}}
\newcommand{\calO}{\mathcal{O}}
\newcommand{\calU}{\mathcal{U}}
\newcommand{\calV}{\mathcal{V}}
\newcommand{\wh}[1]{\widehat{#1}}
\newcommand{\wt}[1]{\widetilde{#1}}

\newcommand{\colim}{\mathrm{colim}}

\newcommand{\fppf}{\mathrm{fppf}}
\newcommand{\crys}{\mathrm{crys}}
\renewcommand{\Im}{\mathrm{Im}}
\renewcommand{\Re}{\mathrm{Re}}

\newcommand{\bbA}{\mathbb{A}}
\newcommand{\bbF}{\mathbb{F}}
\newcommand{\bbG}{\mathbb{G}}
\newcommand{\bbH}{\mathbb{H}}
\newcommand{\bbQ}{\mathbb{Q}}
\newcommand{\bbR}{\mathbb{R}}
\newcommand{\bbV}{\mathbb{V}}
\newcommand{\bbZ}{\mathbb{Z}}
\newcommand{\bbC}{\mathbb{C}}

\newcommand{\frakm}{\mathfrak{m}}

\newcommand{\bighecke}{\mathrm{u}}

\newcommand{\Tate}{\mathrm{Tate}}

\newcommand{\hol}{\mathrm{hol}}
\newcommand{\op}{\mathrm{op}}

\newcommand{\SL}{\mathrm{SL}}

\renewcommand{\Ig}{\mathfrak{I}}
\newcommand{\CS}{\mathrm{CS}}

\renewcommand{\l}{\left}
\renewcommand{\r}{\right}

\newcommand{\bs}{\backslash}

\newcommand{\bD}{\mathbf{D}}
\newcommand{\bE}{\mathbf{E}}

\numberwithin{equation}{subsubsection}
\theoremstyle{plain}

\newtheorem{maintheorem}{Theorem}

\newtheorem*{theorem*}{Theorem}

\newtheorem{theorem}[subsubsection]{Theorem}
\newtheorem{corollary}[subsubsection]{Corollary}

\newtheorem{proposition}[subsubsection]{Proposition}
\newtheorem{lemma}[subsubsection]{Lemma}

\theoremstyle{definition}

\newtheorem{example}[subsubsection]{Example}

\newtheorem{remark}[subsubsection]{Remark}

\title[A unipotent circle action on $p$-adic modular forms]{A unipotent circle action \\ on $p$-adic modular forms}
\author{Sean Howe}
\email{sean.howe@utah.edu}
\address{Department of Mathematics, University of Utah, Salt Lake City, UT~84112}
\thanks{The author was supported during the preparation of this work by the National Science Foundation under Award No. DMS-1704005.}

\begin{document}

\begin{abstract}
Following a suggestion of Peter Scholze, we construct an action of $\widehat{\GG_m}$ on the Katz moduli problem, a profinite-\'{e}tale cover of the ordinary locus of the $p$-adic modular curve whose ring of functions is Serre's space of $p$-adic modular functions. This action is a local, $p$-adic analog of a global, archimedean action of the circle group $S^1$ on the lattice-unstable locus of the modular curve over $\bbC$. To construct the $\wh{\bbG_m}$-action, we descend a moduli-theoretic action of a larger group on the (big) ordinary Igusa variety of Caraiani-Scholze. We compute the action explicitly on local expansions and find it is given by a simple multiplication of the cuspidal and Serre-Tate coordinates $q$; along the way we also prove a natural generalization of Dwork's equation $\tau=\log q$ for extensions of $\bbQ_p/\bbZ_p$ by $\mu_{p^\infty}$ valid over a non-Artinian base. Finally, we give a direct argument (without appealing to local expansions) to show that the action of $\wh{\bbG_m}$ integrates the differential operator $\theta$ coming from the Gauss-Manin connection and unit root splitting, and explain an application to  Eisenstein measures and $p$-adic $L$-functions. 
\end{abstract}

%% Basic setup commands
% If you don't want a title page comment out the next line and uncomment the line after it:
\maketitle
%\omittitle

%% Make the various tables of contents
\setcounter{tocdepth}{1}
\tableofcontents

\section{Introduction and analogy} In this work, following a suggestion of Peter Scholze, we descend the unipotent quasi-isogeny action on a component Caraiani-Scholze's \cite[Section 4]{caraiani-scholze:generic} ordinary (big) Igusa formal scheme for $\GL_2$ to construct an action of the formal $p$-adic torus $\wh{\bbG_m}$ on the Katz moduli problem over the ordinary locus. Suitably interpreted, this action is a local, $p$-adic analog of the global, archimedean phenomena whereby the horizontal translation action of $\bbR$ on the complex upper half plane $\bbH$ descends to an action of $S^1$ on the image of ${\{ \mathrm{Im} \tau > 1\} \subset \bbH}$ in the complex modular curve. 

The space of functions on the Katz moduli problem that are holomorphic at the cusps is equal the completion of classical modular forms for the $q$-expansion topology (Serre's space of $p$-adic modular functions\footnote{In the body of this text, we reserve the term $p$-adic modular forms for those $p$-adic modular functions with a weight, i.e. that transform via a character under the $\bbZ_p^\times$-action.}). Thus we may think of our $\wh{\bbG_m}$-action as a unipotent circle action on $p$-adic modular functions. The analogy with the archimedean circle action is stronger than one might first guess, and leads, e.g.,  to interesting representation-theoretic consequences. 

After constructing the $\wh{\bbG_m}$-action, we study its properties and interaction with other classical notions in the $p$-adic theory of modular curves and modular forms such as the unit root splitting,  Dwork's equation $\tau=\log q$, the differential operator $\theta$, Gouvea's twisting measure, and Katz's Eisenstein measures. 

We highlight one application to explain the significance of this construction: Via $p$-adic Fourier theory, the $\wh{\bbG_m}$-action is equivalent to the $p$-adic interpolation of powers of polynomials in the differential operator $\theta$. This allows us to introduce a twisting direction into any $p$-adic family of modular forms. In particular, when applied to Eisenstein series, it allows the construction of certain two-variable $p$-adic L-functions studied by Katz \cite{katz:p-adic-interpolation-of-real-analytic-eisenstein-series} starting from single-variable Eisenstein families. A key advantage of this method is that we construct the $\wh{\bbG_m}$-action and then relate it to differential operators obtained from the Gauss-Manin connection without ever using a cuspidal or Serre-Tate ordinary local expansion. In particular, our method will generalize to the $p$-adic interpolation of certain differential operators constructed by Eischen and Mantovan \cite{eischen-mantovan:p-adic-familes-mu-ordinary} on the $\mu$-ordinary locus of more general PEL Shimura Varieties (where local expansions are unavailable or difficult to work with) into actions of Lubin-Tate formal groups. 

In the present work we have focused on exploring the ramifications of the existence of a large quasi-isogeny action on the Caraiani-Scholze Igusa formal scheme for the classical space of Katz/Serre $p$-adic modular functions. In a sequel \cite{howe:completed-kirillov}, we study the action of the quasi-isogeny group on the space of functions on the big Igusa formal scheme itself as a natural space of $p$-adic automorphic forms in the context of the $p$-adic Langlands program. Ordinary $p$-adic modular forms (in the sense of Hida) play an important role in this study, and in \cite{howe:completed-kirillov} we also explain how Hida's finiteness and classicality theorems for ordinary $p$-adic modular forms can be understood from this perspective. 

\subsubsection{Acknowledgements}
We thank Ana Caraiani, Ellen Eischen, Matt Emerton, Yifeng Liu, Elena Mantovan, Peter Scholze, and Jesse Silliman for helpful conversations. We thank Jared Weinstein and an anonymous referee for helpful comments and suggestions.

\subsection{An archimedean circle action}
\newcommand{\gord}{\mathrm{\infty-ord}}
Before stating our results, we explain the analogous archimedean circle action more carefully; this will help to motivate and clarify the $p$-adic constructions that follow. Consider the complex manifold
\begin{equation}\label{eqn:arch-unif} Y_{\gord} := \begin{pmatrix}1 & \bbZ \\ 0 & 1\end{pmatrix} \backslash \{ \mathrm{Im} \tau > 1 \}. \end{equation}
Two important observations about $Y_\gord$ follow immediately from (\ref{eqn:arch-unif}):
\begin{enumerate}
\item Modular forms of level $\Gamma_1(N)$ (for any $N$) restrict to $\begin{pmatrix}1 & \bbZ \\ 0 & 1\end{pmatrix}$-invariant functions $\{ \mathrm{Im} \tau > 1 \}$, and thus induce holomorphic functions on $Y_\gord$.
\item The action of $\bbR$ by horizontal translation on $\bbH$ descends to a (real analytic) action of the circle group $S^1$ on $Y_\gord$. This action integrates the vector field $\frac{d}{d\tau}$. 
\end{enumerate}

We can decompose any holomorphic function $f$ on $Y_\gord$ according to this $S^1$ action uniquely as a Fourier series 
\[ f(q) = \sum_{n \in \bbZ} a_n q^n, \; q=e^{2\pi i z}. \]
In other words, the space of functions on $Y_\gord$ is a Fr\'{e}chet completion of the direct sum of the character spaces for the $S^1$-action, with each character appearing exactly once. 

\subsubsection{Fourier coefficients and representation theory}
The Fourier coefficients $a_n$ of classical modular forms play an important role in the global automorphic representation theory for $\GL_2$. In particular, for a Hecke eigenform, the constant coefficient $a_0$ is non-vanishing if and only if the corresponding global automorphic representation is globally induced (i.e. the modular form is Eisenstein). Suitably interpreted, the constant term $a_0$ is a functional that realizes the induction. The non-constant coefficients, on the other hand, are Whittaker functionals.

\subsubsection{The slope formalism on metrized tori} 
While the construction of $Y_\gord$ above may at first seem ad hoc, it has a natural moduli interpretation, which we explain now. The key point is to use the slope formalism for metrized tori, or, equivalently, lattices, as explained, e.g., in Casselman's survey \cite{casselman:stability}. 

A metrized torus is a finite dimensional torus (compact real abelian Lie group) $T$ together with a translation invariant metric, or, equivalently, a positive definite inner product on $\Lie T \cong H_1(T, \bbR)$. There is a natural slope formalism on metrized tori: the rank function is dimension, and the degree function is given by
\[ \deg T := \mathrm{log} \mathrm{Vol}(T). \]
If a two-dimensional metrized torus $T$ is unstable (i.e., not semi-stable), then it contains a unique circle of shortest length. 

If $E/\bbC$ is an elliptic curve, the underlying real manifold of $E(\bbC)$ is a two-dimensional metrized torus when equipped with the metric coming from the canonical principal polarization. 

\begin{example}\label{example:arch-unif}
Consider the usual fundamental domain 
\[ \calD:= \{ \tau \in \bbH, -1/2 \leq \Re \tau \leq 1/2 \textrm{ and } |\tau|\geq 1\}, \]
for the action of $\SL_2(\bbZ)$ on $\bbH$. For $\tau \in \calF$, let
\[ E_\tau:=\bbC / \langle 1, \tau \rangle. \]

We compute the values of $\tau \in \calD$ for which $E_\tau$ is semistable: The metric induced by the principal polarization is identified with $1/\Im\tau$ times the metric induced by the identity 
\[ \bbR 1 + \bbR \tau = \bbC \]
and the standard metric on $\bbC$. Semistability is preserved by scaling the metric, so we may eliminate the scaling and consider just the metric induced by the standard metric on $\bbC$. The length of a shortest circle in $E_\tau(\bbC)$ is equal to the length of a shortest vector in $H_1(E_\tau, \bbZ)$, which is $1$. The area of the entire torus $E_\tau(\bbC)$, on the other hand, is $\Im \tau$.  Thus, the slope of the full torus is $\frac{1}{2}\log \Im \tau$, while the smallest slope of a circle inside is $0$.  We conclude that for $\tau \in \calD$, $E_\tau$ is semi-stable when $\Im \tau \leq 1$, and otherwise is unstable with shortest circle given by 
\[ S^1 = \bbR / \bbZ \hookrightarrow \bbC / \langle 1, \tau \rangle. \]
\end{example}

\subsubsection{Moduli of unstable elliptic curves}
Using the slope formalism, we may consider the moduli space of unstable elliptic curves $E/\bbC$ equipped with a trivialization of the shortest circle, $S^1 \hookrightarrow E(\bbC)$. From Example \ref{example:arch-unif}, we find that this space is naturally identified with $Y_\gord$. In this moduli interpretation, the space $\Im \tau > 1$ is the cover where the trivialization of the shortest circle is extended to an oriented trivialization $S^1 \times S^1 \xrightarrow{\sim} E(\bbC)$. From the moduli perspective, the fact that we can evaluate modular forms to obtain functions on $Y_\gord$ comes from two facts:
\begin{enumerate}
\item Given a point of $Y_\gord$, there is a unique holomorphic differential $\omega_\can$ whose pullback to $S^1$ along the trivialization of the shortest circle integrates to $1$. Thus, the modular sheaf $\omega$ is canonically trivialized over $Y_\gord$, and modular forms can be evaluated along this trivialization,. 
\item Using the polarization, the trivialization of the shortest circle also gives rise to a trivialization of the quotient torus $E(\bbC)/S^1$, so that $E(\bbC)$ is equipped with the structure of an extension of real tori
\begin{equation}\label{equation:arch-extension} 1 \rightarrow S^1 \rightarrow E(\bbC) \rightarrow S^1 \rightarrow 1 \end{equation}
The basis $1/N$ for the torsion on $S^1=\bbR/\bbZ$ then gives rise to a canonical $\Gamma_1(N)$-level structure on $E$ for any level $N$.
\end{enumerate}

\subsubsection{de Rham cohomology}
Consider the extension structure (\ref{equation:arch-extension}) on the universal elliptic curve over $Y_\gord$. The global section $dx$ of the de Rham cohomology of $S^1=\bbR/\bbZ$ is flat, so we obtain via pullback of $dx$ a canonical flat section $u_\can \in H^1_\dR(E(\bbC), \bbR)$ over $Y_\gord$. Moreover, because the image of $\omega_\can$ in $H^1_\dR(S^1, \bbC)$ under pullback is (by definition) $dx$, which is flat, we find that $\nabla(\omega_\can)$ is in the span of $u_\can$. Thus, $\frac{\nabla \omega_\can}{u_\can}$ is a holomorphic differential form on $Y_\gord$. 

For the elliptic curve $E_\tau$ as in Example \ref{example:arch-unif}, if we denote by $e_1$ and $e_\tau$ the natural basis elements for $H_1(E(\bbC),\bbZ)$ and by $e_1^*$ and $e_\tau^*$ the dual basis, we find that $\omega_\can = e_1^* + \tau e_\tau^*$, and $u_\can = e_\tau^*$, so that 
\[ \frac{\nabla \omega_\can}{u_\can}=d\tau = d\log q. \]
In particular, the $S^1$ action integrates the vector field $\frac{d}{d\tau}$ dual to $\frac{\nabla \omega_\can}{u_\can}$. 

\subsection{Statement of results}
In this section we state our main results. 

\subsubsection{Dictionary} As we introduce the objects appearing in the local, $p$-adic theory, it may be helpful to keep in mind the following dictionary for our analogy with the global, archimedean story:\\
\def\arraystretch{1.5}
\begin{tabular}{p{2.25in} | p{2.25in}}
\textit{Global, archimedean }& \textit{Local, $p$-adic} \\
\hline
$E(\bbC)$ as a metrized torus & The $p$-divisible group $E[p^\infty]$ \\
\hline
Unstable two-dimensional \newline metrized torus & Ordinary height two $p$-divisible group \\
\hline
The shortest circle in $E(\bbC)$ & The formal group $\wh{E}$ \\
\hline
Trivialization of the shortest circle\newline $S^1 \hookrightarrow E(\bbC)$ & Trivialization of the formal group \newline $\wh{E}  \xrightarrow{\sim} \wh{\bbG_m}$ \\
\hline
$1 \rightarrow S^1 \rightarrow E(\bbC) \rightarrow S^1 \rightarrow 1$ & $1 \rightarrow \wh{\bbG_m} \rightarrow E[p^\infty] \rightarrow \bbQ_p/\bbZ_p \rightarrow 1.$ \\
\hline
$d\tau$, $\omega_\can$, $u_\can$ & $d\tau$, $\omega_\can$, $u_\can$\\
\hline
Canonical $\Gamma_1(N)$ level structure & Canonical arithmetic $\Gamma_1(p^n)$ \newline level structure \\
\hline
$Y_\gord$ & The Katz formal scheme $\Ig_\Katz$ \\
\hline
$\Im \tau \geq 1$ & The (polarized) Igusa formal scheme of Caraiani-Scholze $\Ig_{\CS}^1$ \\
\hline
Action of $S^1$ on $Y_\gord$ & Action of $\wh{\bbG_m}$ on $\Ig_\Katz$ \\
\hline
Action of $\bbR$ on $\Im \tau \geq 1$  & Action of the universal cover of $\wh{\bbG_m}$ on $\Ig_{\CS}^1$ \\
\hline
Fourier series & Sheaf over $\bbZ_p$ \\
\hline 
Constant term & Fiber at 0 \\
\hline
Eisenstein series & Ordinary $p$-adic modular form \\

\end{tabular}

\subsubsection{The Katz moduli problem and $p$-adic modular forms}
For $R$ a $p$-adically complete ring, let $\Nilp_R$ be the category of $R$-algebras in which $p$ is nilpotent. We consider the Katz moduli problem on $\Nilp_{\bbZ_p}^{\op}$ classifying, over $S \in \Nilp_{\bbZ_p}$, triples
\[ (E, \wh{\varphi}, \varphi_{\bbA_f^{(p)}}) \]
where $E/\Spec S$ is an elliptic curve up to prime-to-$p$ isogeny, $\wh{\varphi}$ is a trivialization of the formal group of $E$,
\[ \wh{\varphi} : \wh{E} \xrightarrow{\sim} \wh{\bbG_m}, \]
and $\varphi_{\bbA_f^{(p)}}$ is a trivialization of the adelic prime-to-$p$ Tate module.  

By work of Katz \cite{katz:l-via-moduli}, the moduli problem is represented by a $p$-adic formal scheme 
\[ \Ig_{\Katz} = \Spf \bbV_{\Katz} \]
where $\bbV_{\Katz}$ is a $p$-adically complete flat (torsion-free) $\bbZ_p$-algebra. For $R$ a $p$-adically complete $\bbZ_p$-algebra, we write 
\[ \bbV_{\Katz,R} := \bbV_{\Katz} \wh{\otimes}_{\bbZ_p} R \]
so that 
\[ \Ig_{\Katz,R} := \Ig_{\Katz} \times_{\Spf \bbZ_p} \Spf R = \Spf \bbV_{\Katz, R}. \]

There is a natural moduli action of $\bbZ_p^\times \times \GL_2\l(\AA_f^{(p)}\r)$ on $\Ig_\Katz$, where ${\bbZ_p^\times=\Aut(\wh{\bbG_m})}$ acts by composition with $\wh{\varphi}$, and $\GL_2(\AA_f^{(p)})$ acts by composition with $\varphi_{\bbA_f^{(p)}}$. For a continuous character $\kappa$ of $\bbZ_p^\times$ with values in $R$, the eigenspace $\VV_{\Katz, R}[\kappa]$ is a natural space of $p$-adic modular forms of weight $\kappa$.

In particular, classical modular forms of integral weight and prime-to-$p$ level over $\bbZ_p$ (interpreted using sections of the standard integral model of the modular sheaf and curve) are embedded $\GL_2(\AA_f^{(p)})$-equivariantly (up to a twist) in this space for the character $z \mapsto z^k$. Concretely, a classical modular form of prime-to-$p$ level $K^p$ and weight $k$ gives rise to an element of $\bbV_{\Katz}^{K^p}[k]$ by evaluating on the triple
\[ \l(E_\univ, \wh{\varphi}_{\univ}^* \l(\frac{dt}{t}\r), \varphi_{\bbA_f^{(p)},\univ}/K^p\r). \]
Here $(E_\univ, \wh{\varphi}_{\univ}, \varphi_{\bbA_f^{(p)},\univ})$ is the universal triple parameterized by the identity map on $\Ig_{\Katz}$ and $\frac{dt}{t}$ is an invariant differential form on $\wh{\bbG_m}$. In other words, the modular sheaf $\omega$ on $\Ig_{\Katz}$ is trivialized by the canonical section
\[ \omega_\can := \wh{\varphi}_{\univ}^*\l(\frac{dt}{t}\r) \]
which allows us to evaluate classical modular forms after pullback to $\Ig_{\Katz}$.

\subsubsection{de Rham cohomology}
\newcommand{\bbD}{\mathbb{D}}
We write 
\[ \pi:E_\univ \rightarrow \Ig_{\Katz},\]
for the universal elliptic curve up-to-prime-to-$p$-isogeny. We have the relative de Rham cohomology
\[ H^1_\dR(E_\univ):=R^1 \pi_* \Omega^\bullet_{E_\univ/\Ig_\Katz} \]
equipped with Hodge filtration 
\[ 0 \rightarrow \omega_E \rightarrow H^1_\dR \rightarrow \Lie E^\vee \rightarrow 0 \]
and Gauss-Manin connection $\nabla$. 

Note that the moduli problem classified by $\Ig_\Katz$ is equivalent to the moduli problem classifying triples $(E, \wh{\varphi}, \varphi_{\wh{\bbZ}^{(p)}})$ where $E$ and $\wh{\varphi}$ are as before, and $\varphi_{\wh{\bbZ}^{(p)}}$ is a trivialization of the prime-to-$p$ Tate module
\[ \varphi_{\wh{\bbZ}^{(p)}}: (\wh{\bbZ}^{(p)})^2 \xrightarrow{\sim} T_{\wh{\bbZ}^{(p)}} E = \lim_{(n,p)=1} E[n],\]
all considered up to \emph{isomorphism} of $E$. Using this equivalence, we obtain a well-defined Weil pairing on $E[p^\infty]$, and combining this with the trivialization $\wh{\varphi}$, we obtain the structure of an extension
\begin{equation}\label{equation:intro-igusa-extension} 1 \rightarrow \wh{\bbG_m} \rightarrow E[p^\infty] \rightarrow \bbQ_p/\bbZ_p \rightarrow 1. \end{equation}
This is analogous to the archimedean extension (\ref{equation:arch-extension}). In particular, we obtain an extension of Dieudonn\'{e} crystals 
\[ 0 \rightarrow \bD(\bbQ_p/\bbZ_p) \rightarrow \bD(E_\univ[p^\infty]) \rightarrow \bD(\wh{\bbG_m}) \rightarrow 0,\]
which we view as an extension of vector bundles with connection on $\Ig_\Katz$. The sub-bundle $\bD(\bbQ_p/\bbZ_p)$ is the unit root filtration, and it has a canonical basis element $u_\can$. If we identify $\bD(E_\univ[p^\infty])$ with $(H^1_{\dR}(E_\univ), \nabla)$ via the crystalline-de Rham comparison, then the Hodge filtration splits the extension of vector bundles and the image of $\omega_\can$ in $\bD(\wh{\bbG_m})$ is flat. So, $\omega_\can$ and $u_\can$ give a basis for $H^1_\dR(E_\univ)$ such that $\nabla$ is lower nilpotent, and thus determined by a single differential form 
\[ d\tau := \frac{\nabla(\omega_\can)}{u_\can}. \]
By the theory of Kodaira-Spencer, the differential form $d\tau$ is non-vanishing, and thus admits a dual vector field $\frac{d}{d\tau}$ such that $\langle d\tau, \frac{d}{d\tau} \rangle=1$. 

\subsubsection{The $\wh{\bbG_m}$-action.}
Our main result, Theorem \ref{maintheorem:big-hecke} below, shows that the vector field $\frac{d}{d\tau}$ can be integrated to an action of $\wh{\bbG_m}$ on $\Ig_\Katz$, and explains how this action interacts with the action of $\bbZ_p^\times \times \GL_2(\bbA_f^{(p)}).$ To state it, we will need the unramified determinant character $\det_\ur: \GL_2(\bbA_f^{(p)}) \rightarrow \bbZ_{(p)}$ defined by
\[ \det_\ur( (g_l)_{l\neq p} ) = \prod_{l \neq p} |\det g_l|_l. \]

\begin{maintheorem}\label{maintheorem:big-hecke}
There is an action of $\wh{\bbG_m}$ on $\Ig_{\Katz}$ whose derivative is the vector field $\frac{d}{d\tau}$ defined above. Moreover, this action combines with the action of $\bbZ_p^\times \times \GL_2(\AA_f^{(p)})$ to give an action of 
\[ \wh{\bbG_m} \rtimes (\bbZ_p^\times \times \GL_2(\AA_f^{(p)})) \]
where the semi-direct product is formed with the respect to the conjugation action
\[ (z,g) \cdot \zeta \cdot (z, g)^{-1} = \zeta^{z^2 \det_\ur(g)}. \]
\end{maintheorem}

\begin{remark} The $\wh{\bbG_m}$-action of Theorem \ref{maintheorem:big-hecke} is uniquely determined by the condition that it integrates $\frac{d}{d\tau}$. We note that $\frac{d}{d\tau}$ acts as the derivation $-\theta=-q\frac{d}{dq}$ on cuspidal $q$-expansions; however, in our proof we construct the action and prove it integrates $\frac{d}{d\tau}$ without using cuspidal (or Serre-Tate) $q$-expansions, which is an important point for future generalizations. 
\end{remark}

The key observation in the construction of this $\wh{\bbG_m}$-action and subsequent computations is that we may work on a very ramified cover, a component of the (big) Igusa formal scheme of Caraiani-Scholze \cite[Section 4]{caraiani-scholze:generic}, where the extension structure (\ref{equation:intro-igusa-extension})  extends to a trivialization of the $p$-divisible group 
\[ \varphi: E[p^\infty] \xrightarrow{\sim} \widehat{\GG}_m \times \bbQ_p/\bbZ_p. \]
At the price of the ramification, life is simplified on this cover: for example, computations with the crystalline connection are reduced to computing the crystalline realization of maps $\bbQ_p/\bbZ_p \rightarrow \wh{\bbG_m}$. Most importantly, the obvious action of automorphisms of $\wh{\bbG_m} \times \bbQ_p/\bbZ_p$ on this cover extends to an action of a much larger group of quasi-isogenies.

This quasi-isogeny group contains a very large unipotent subgroup, the quasi-isogenies from $\bbQ_p/\bbZ_p$ to $\wh{\bbG_m}$, or, the universal cover $\wt{\wh{\bbG_m}}$ in the language of Scholze-Weinstein \cite{scholze-weinstein:p-div}. The action of this large group of quasi-isogenies is the ultimate source of the $\wh{\bbG_m}$-action on $\Ig_\Katz$. Indeed: $\Ig_\Katz$ is the quotient by the subgroup of isomorphisms, i.e. the Tate module $T_p \wh{\bbG_m}$, and thus picks up a residual action of 
\[ \wh{\bbG_m}=\wt{\wh{\bbG_m}} / T_p \wh{\bbG_m}. \]

\begin{remark}
The action of a larger group of quasi-isogenies on this cover is a natural characteristic $p$ analog of the prime-to-characteristic phenomenon where, when full level is added at $l\neq p$, there is an isogeny moduli interpretation that gives an action of $\mathrm{GL}_2(\bbQ_l)$ extending the action of $\mathrm{GL}_2(\bbZ_l)$ in the isomorphism moduli interpretation. Rigidifying in characteristic $p$ using isomorphisms to an ordinary $p$-divisible group provides both more and less structure than when $l \neq p$: on the one hand, the isogeny group is solvable, and thus appears more like the subgroup of upper triangular matrices, but on the other hand the unipotent subgroup has a much richer structure than any groups that appear when $l \neq p$. If we instead rigidified using a height two formal group, we would obtain a super-singular Igusa variety, which has more in common with the $l \neq p$ case (the isogeny action is by the invertible elements of the non-split quaternion algebra over $\bbQ_p$); in \cite{howe:p-adic-jl} we use this structure to compare $p$-adic modular forms and continuous $p$-adic automorphic forms on the quaternion algebra ramified at $p$ and $\infty$. 
\end{remark}

\begin{remark}\label{remark:big-hecke-terminology} In this remark we explain a connection to perfectoid modular curves: The generic fiber of the big Igusa formal scheme is a twist of a component of the perfectoid ordinary locus over $\bbC_p$. This component admits a natural action of the group of upper triangular matrices
\[ \begin{pmatrix} 1 & \bbQ_p \\ 0 & 1 \end{pmatrix} \]
which is identified over $\bbC_p$ with an action of $\bbQ_p(1)$ on the generic fiber of the big Igusa formal scheme.

Using this, the action of the $p$-power roots of unity $\QQ_p(1)/\ZZ_p(1)$, an infinite discrete set inside of the open ball $\wh{\bbG_m}(\calO_{\bbC_p})$, on functions on the generic fiber of $\Ig_{\Katz, \calO_{\bbC_p}}$ can be identified with the action of the natural Hecke operators $\QQ_p/\ZZ_p$ on the invariants under 
\[ \begin{pmatrix} 1 & \ZZ_p \\ 0 & 1 \end{pmatrix} \] in functions on this component of the perfectoid ordinary locus. Thus, the $\wh{\bbG_m}$ action extends the obvious action of $\bbQ_p/\bbZ_p$ to an action of a much larger group. We will not use this connection to perfectoid modular curves in the present work, however, it will play an important role in \cite{howe:completed-kirillov}.     
\end{remark}

\subsubsection{Local expansions}
An important aspect of our proof of Theorem \ref{maintheorem:big-hecke} is that we make no appeal to local expansions at cusps or ordinary points, so that our approach is well-suited for generalization to other PEL Igusa varieties. After proving Theorem~\ref{maintheorem:big-hecke}, however, we also give a direct computation of the action on local expansions: we find that at ordinary points the action is given by multiplication of a Serre-Tate coordinate, and at the cusps it is given by multiplication of the inverse of the standard cuspidal coordinate $q$.

\subsubsection{Dwork's equation.} While developing some of the machinery used to compute the local expansions of the $\wh{\bbG_m}$-action, and using the same philosophy of base change to a very ramified cover, we also give a new proof of Dwork's equation $\tau=\log q$ on the formal deformation space of $\wh{\bbG_m} \times \bbQ_p/\bbZ_p$ over $\overline{\bbF_p}$ which is valid for a larger family of \emph{Kummer} $p$-divisible groups (which include not only the deformations of $\wh{\bbG_m} \times \bbQ_p/\bbZ_p$ over Artinian $\overline{\bbF_p}$-algebras, but also, e.g., the $p$-divisible group of the Tate curve, and other interesting groups when the base is not Artinian). These results can be found in Section \ref{section:extensions}. 

\subsubsection{Other constructions}
This action can be constructed in at least three other ways, two of which have been discussed previously in the literature:
\begin{enumerate}
\item After preparing an earlier version of this article, we learned that Gouvea \cite[III.6.2]{gouvea:arithmetic-of-p-adic-modular-forms} had already some time ago constructed a twisting measure equivalent to our $\wh{\bbG_m}$-action (interpreted as an algebra action via $p$-adic Fourier theory as described in \ref{sss:intro-algebra} below). In \ref{ss:gouvea} we recall Gouvea's construction and explain how it can be rephrased as an alternate construction of the $\wh{\bbG_m}$-action via the exotic isomorphisms of Katz \cite[5.6]{katz:p-adic-interpolation-of-real-analytic-eisenstein-series}.
	Gouvea's construction has the advantage of using only classical ideas, but is conceptually more opaque. In particular, we note that the interaction of the $\wh{\bbG_m}$-action with the prime-to-$p$ group action (equivalently, Hecke operators away from $p$) is considerably clarified by our construction. 

\item While the current version of this article was under review, we learned that our $\wh{\bbG_m}$-action is also a special case of a result of Liu-Zhang-Zhang \cite[Proposition 2.3.5]{liu-zhang-zhang}, who gave a construction of a Lubin-Tate action on more general Shimura curves by using the Baer sum of extensions. There are some issues with the proof as written in loc. cit. because of a mistake in the statement of Serre-Tate theory \cite[Theorem B.1.1]{liu-zhang-zhang} over rings where $p$ is nilpotent (where one must allow for unipotent quasi-isogenies as well as isomorphisms, consider e.g. the base $\bbF_p[x]/x^2$). The connection between these two constructions will be elaborated further in future work of the author constructing actions on $\mu$-ordinary Igusa varieties. 

\item The simplest and most opaque approach is to build the $\wh{\bbG_m}$-action algebraically starting with the differential operator $\theta$ and the $q$-expansion principle; we explain this in Remark \ref{remark:big-hecke-theta-construction} below. 
\end{enumerate}

\subsubsection{The algebra action}\label{sss:intro-algebra}
Via $p$-adic Fourier theory, the action of $\wh{\bbG_m}$ described in Theorem \ref{maintheorem:big-hecke} is equivalent to an action of $\Cont(\bbZ_p, \bbZ_p)$ on $\bbV_\Katz$. This action admits a particularly simple description on cuspidal $q$-expansions: $f \in \Cont(\bbZ_p, \bbZ_p)$ acts as multiplication by $f(n)$ on the coefficient of $q^n$ (cf. Theorem \ref{theorem:algebra-action-q-expansions}). As remarked above, the existence of this algebra action was first established by Gouvea \cite[Corollary III.6.8]{gouvea:arithmetic-of-p-adic-modular-forms}, who interpreted it as a twisting measure. 

From this perspective, the action of the monomial function $z^k$ is by the derivation $\theta^k$ (recall $\theta=q\frac{d}{dq}$), and thus we may view our $\wh{\bbG_m}$-action as interpolating the differential operators $\theta^k$ into an algebra action. In Section \ref{section:eisenstein-measures} we adopt this perspective to reinterpret some results of Katz \cite{katz:p-adic-interpolation-of-real-analytic-eisenstein-series} on two-variable Eisenstein measures.

\begin{remark}\label{remark:big-hecke-theta-construction}
In fact, we can \emph{construct} the $\wh{\bbG_m}$-action by applying the $q$-expansion principle \cite[5.2]{katz:p-adic-interpolation-of-real-analytic-eisenstein-series} to complete the action of polynomials in $\theta$ on $\bbV_\Katz$ to an action of $\Cont(\bbZ_p, \bbZ_p)$. Note that polynomials are not dense $\Cont(\bbZ_p,\bbZ_p)$, so the $q$-expansion principle needed here says not just that the $q$-expansion map is injective, but also that the cokernel is flat over $\bbZ_p$. 

In order to use this method, one must first show that the operator $\theta$ on $q$-expansions preserves the space of $p$-adic modular forms (instead of deducing this by differentiating the $\wh{\bbG_m}$ action). One way this can be done is by showing it is the effect on $q$-expansions of the differential operator dual to the image of $\omega_\can^2$ under the Kodaira-Spencer isomorphism, which can be verified by a computation over $\bbC$, as explained by Katz \cite[5.8]{katz:p-adic-interpolation-of-real-analytic-eisenstein-series}. 
\end{remark}

\subsubsection{Ordinary $p$-adic modular forms}
The action of $\Cont(\bbZ_p, \bbZ_p)$ interacts naturally with the $\bbZ_p^\times$ action on $\bbV_{\Katz}$, and thus  we may view $\bbV_\Katz$ as a $\bbZ_p^\times$-equivariant quasi-coherent sheaf on the profinite set $\bbZ_p$ (viewed as a formal scheme whose ring of functions \emph{is} $\Cont(\bbZ_p,\bbZ_p)$). As $\bbZ_p$ is the space of characters of $\wh{\bbG_m}$, this viewpoint is analogous to thinking of functions on $Y_\gord$ in the global, archimedean setting as Fourier series. 

A straightforward computation with $q$-expansions implies that restriction induces an isomorphism between the fiber at $0 \in \bbZ_p$ of the subsheaf $\bbV_{\Katz,\hol}$ of $p$-adic modular function with $q$-expansion holomorphic at all cusps and the space of ordinary $p$-adic modular forms \`{a} la Hida. Note that the fiber at zero is the maximal trivial quotient for the $\wh{\bbG_m}$-action, and ordinary modular forms are those such that the corresponding $p$-adic Banach representation of $\GL_2(\bbQ_p)$ admits a map to a unitary principal series. Thus, our statement is a local, $p$-adic analog of the global, archimedean statement that the global automorphic representation attached to a classical modular form is globally induced if and only if its Fourier expansion has a non-zero constant term. 

We do not discuss this phenomenon further in the present work, but this characterization of ordinary $p$-adic modular forms will play an important role in our study of functions on $\Ig_\CS$ as a natural space of $p$-adic automorphic forms in \cite{howe:completed-kirillov}. Moreover, this perspective also leads to representation-theoretic proofs of Hida's finiteness and classicality results for ordinary $p$-adic modular forms, as will be explained in \cite{howe:completed-kirillov}.  

\subsection{A remark on notation}
Over a ring in which $p$ is topologically nilpotent, the formal group $\wh{\bbG_m}$ is equivalent to the $p$-divisible group $\mu_{p^\infty}$. In the introduction so far we have only used the notation $\wh{\bbG_m}$, because we wanted to emphasize in our discussion of the action that this is not a torsion group (e.g., the $\bbZ_p$-points are $1+p\bbZ_p$). In the remainder of the article, however, it will be convenient to prefer the notation $\mu_{p^\infty}$ when we are speaking about $p$-divisible groups appearing, e.g., in a moduli problem, and to generally reserve the notation $\wh{\bbG_m}$ for when we are discussing the action on $\Ig_{\Katz}$. This is especially convenient to avoid the oversized notation
\[ \wt{\wh{\bbG_m}} \]
when discussing universal covers! 

\subsection{Outline}
In Section \ref{section:p-divisible-groups} we collect some results on $p$-divisible groups that will be needed in the rest of the paper. In Section \ref{section:extensions} we study extensions of $\bbQ_p/\bbZ_p$ by $\mu_{p^\infty}$; in particular, we introduce \emph{Kummer $p$-divisible groups} (following a construction of Katz-Mazur \cite[8.7]{katz-mazur}) and prove our generalization of Dwork's formula ${\tau=\log q}$. 

In Section \ref{section:moduli-problems} we recall the Katz and Caraiani-Scholze moduli problems over the ordinary locus, and explain the relation between them. 

In Section \ref{sec:the-action} we construct the action of $\wh{\bbG_m}$ and prove Theorem \ref{maintheorem:big-hecke}. In Section~\ref{sec:local-expansions} we compute the action on local expansions, and show that there is no global Serre-Tate coordinate on $\Ig_\Katz$ (dispelling some myths in the literature).  

In Section \ref{sec:algebra-action} we explain how to obtain the algebra action of $\Cont(\bbZ_p, \bbZ_p)$ using $p$-adic Fourier theory, and compare our construction to Gouvea's original construction of this algebra action.  Finally, in Section \ref{section:eisenstein-measures} we explain an application to Eisenstein measures and $p$-adic $L$-functions.

\section{Preliminaries on $p$-divisible groups}\label{section:p-divisible-groups}
\newcommand{\Alg}{\mathrm{Alg}}

In this section we collect some results on $p$-divisible groups that will be useful in our construction. Our principal references are \cite{messing:crystals-associated-to-barsotti-tate} and \cite{scholze-weinstein:p-div}; we also provide some complements. 

For the proof Theorem \ref{maintheorem:big-hecke}, the most important result in this section is Lemma~\ref{lemma:vector-extension-map}. It computes, for $I$ a nilpotent divided powers ideal in a ring $R$ where $p$ is nilpotent, the action of 
\[ \Hom(\bbQ_p/\bbZ_p |_{R/I}, \mu_{p^\infty}|_{R/I}) \]
on the Messing crystals evaluated on $R$.

\subsection{$p$-divisible groups}
Let $R$ be a ring. A $p$-divisible group $G$ of height $h$ over $R$ is, following Tate \cite[2.1]{tate:p-divisible-groups}, an inductive system \[ (G_i, \iota_i), i \geq 0 \]
of finite and locally free group schemes $G_i$ of order $p^{ih}$ over $R$ equipped with closed immersions $\iota_i: G_i \rightarrow G_{i+1}$ identifying $G_i$ with the kernel of multiplication by $p^i$ on $G_{i+1}.$

\begin{example}\hfill
\begin{enumerate}
\item We write $\mu_{p^\infty}$ for the inductive system $(\mu_{p^i})_{i \geq 0}$, where $\mu_{p^n}$ is the kernel of multiplication by $p^n$ on $\bbG_m$, and the inclusion maps are the obvious ones; it is a p-divisible group of height $1$. When $p$ is topologically nilpotent on $R$ we also write $\mu_{p^\infty}=\wh{\bbG_m}$, notation that will be explained below.
\item We write $\bbQ_p/\bbZ_p$ for the inductive system $(1/p^n \bbZ_p / \bbZ_p)$ with the obvious inclusions; it is a $p$-divisible group of height $1$. 
\end{enumerate}
\end{example}

Given a $p$-divisible group, each of the $G_i$ defines a presheaf in abelian groups on $\Alg_R^\op$, and we will also denote by $G$ the presheaf $\colim G_i$ so that 
\begin{equation}\label{equation:p-divisible-group-presheaf}G(S) = \colim G_i(S) \end{equation}
for $S$ an $R$-algebra. With this notation, we have a canonical identification ${G_i=G[p^i]}$. 

\begin{remark}
Note that the maps are injective as maps of presheaves, so that in any faithful topology where the objects of $\Alg_R^\op$ are all quasi-compact (e.g. fppf), (\ref{equation:p-divisible-group-presheaf}) is also the colimit as sheaves by \cite[Lemma 7.17.5]{stacks-project}. In particular, one could instead define a $p$-divisible group as, e.g., an fppf sheaf satisfying certain properties, as is often done in the literature. We prefer the given definition because we will have occasion later on to consider finer topologies. 
\end{remark}

\begin{remark}
We will usually consider $p$-divisible groups over a ring $R$ where $p$ is nilpotent, or over an affine formal scheme $\Spf R$ where $p$ is topologically nilpotent in $R$. In the latter case, there are two natural ways one might try to define $G(S)$ for $S$ a topological $R$-algebra: one could first algebraize to obtain a $p$-divisible group over $\Spec R$, then apply the definition above, or one could take the limit of $G(S/I)$ where $I$ runs over the ideals defining the topology on $R$. The latter is the correct definition for our purpose. For example, if $R=\calO_{\bbC_p}$ with the $p$-adic topology and $G=\wh{\bbG_m} (= \mu_{p^\infty} )$, then, the second, correct, definition gives $\wh{\bbG_m}(R)=1 + \frakm$ where $\frakm$ is the maximal ideal in $\calO_{\bbC_p}$ while the first, incorrect, definition gives only the $p$-power roots of unity.  
\end{remark}

\subsection{Formal neighborhoods and Lie algebras}\label{subsection:formal-neighborhoods-and-lie-algebras}
For $G$ a presheaf in abelian groups on $\Alg_R^\op$, we define the formal neighborhood of the identity $\widehat{G}$ by 
\[ \wh{G}(S) = \ker G(S) \rightarrow G(S^\mathrm{red}) \]
and the Lie algebra $\Lie G$ by
\[ \Lie G(S) = \ker G(S[\epsilon]/\epsilon^2) \rightarrow G(S). \]
Note that, by definition $\Lie G(S) = \Lie \wh{G}(S)$. We have the following important structural result: 
\begin{theorem}\label{theorem:p-divisible-formally-smooth}\cite[Theorems 3.3.13 and 3.3.18]{messing:crystals-associated-to-barsotti-tate} 
If $G$ is a $p$-divisible group over a ring $R$ where $p$ is nilpotent, then $\wh{G}$ is a formal Lie group and $G$ is formally smooth. 
\end{theorem}
For a $p$-divisible group $G$, we will sometimes write $G^\circ$ (the connected component of the identity in $G$) instead of $\wh{G}$ to lighten notation. 

\subsection{Universal covers}\label{subsection:universal-covers}
For any presheaf in abelian groups $G$, we define
\[ \widetilde{G} := \lim G \xleftarrow{p} G \xleftarrow{p} \ldots  \]
and its sub-functor
\[ T_p G := \lim {1} \xleftarrow{p} G[p] \xleftarrow{p} G[p^2] \ldots. \]
For $A \in \Nilp_R^\op$ we will write an element of $\wt{G}(A)$ as a sequence $(g_0, g_1, \ldots )$ such that $p(g_{i+1})=g_i$ for all $i \geq 0$; the elements of $T_p G$ are those such that $g_0=1$. In particular, we have an exact sequence of presheaves
\[ 1 \rightarrow T_p G \rightarrow \widetilde{G} \rightarrow G \]
where the map $\wt{G} \rightarrow G$ is $(g_0, g_1, \ldots) \mapsto g_0$. 

When $G$ is a $p$-divisible group, we call $\widetilde{G}$ the universal cover, following \cite{scholze-weinstein:p-div}. In this case, we have
\begin{lemma}\label{lemma:universal-cover-fpqc-surjection}
If $G$ is a $p$-divisible group,
\begin{equation}\label{equation:universal-cover-exact} 1 \rightarrow T_p G \rightarrow \widetilde{G} \rightarrow G  \rightarrow 1
\end{equation}
is an exact sequence of sheaves in the fpqc topology.
\end{lemma}
\begin{proof}
We must verify that $\wt{G} \rightarrow G$ is surjective as a map of fpqc sheaves. Note that if $G_i = \Spec R_i$, then $\wt{G} \times_G G[p^n]$ is represented by $\Spec\,\colim_{i\geq n} R_i$, and the inclusion $R_n \rightarrow \colim R_i$ is an fpqc cover. Given an $S$-point $f: \Spec S \rightarrow \wt{G}$, which factors through $G[p^n]$ for some $n$, we find $\wt{S}=f^* \wt{G}$ is an fpqc cover of $\Spec S$ such that $f$ is in the image of $\wt{G}(\wt{S})$. 
\end{proof}
\begin{remark} 
Exactness at the right in (\ref{equation:universal-cover-exact}) typically fails in the fppf topology. For example, if $G=\mu_{p^\infty}$ and $R$ is finitely generated of characteristic $p$, then $\widetilde{\mu_{p^\infty}}(R)=1$.  Any fppf cover of such an $R$ is by finitely generated rings of characteristic $p$, thus $\widetilde{\mu_{p^\infty}}$ is the trivial sheaf on the small fppf site of $\Spec R$.  On the other hand, if $R$ contains any nilpotents (e.g. $R=k[\epsilon]/\epsilon^2$), then $\mu_{p^\infty}(R) \neq 1$, and thus the map $\widetilde{\mu_{p^\infty}} \rightarrow \mu_{p^\infty}$ is not surjective in the fppf topology. 
\end{remark}

\subsubsection{Crystalline nature of the universal cover}

Suppose $G_0$ is a $p$-divisible group over a ring $R$ in which $p$ is nilpotent, $R' \rightarrow R$ is a nilpotent thickening, and $G$ is a lift of $G_0$ to $R'$. Then, the reduction map
\[ \wt{G}(R') \rightarrow \wt{G_0}(R) \]
is an isomorphism: the inverse sends $(g_0, g_1, \ldots )$ to
$(g_0', g_1', \ldots)$ where $g_i'$ is defined to be $p^n (\tilde{g}_{i+n})$ for $n$ sufficiently large and any lift $\tilde{g}_{i+n} \in G(S)$ of $g_{i+n}$. These lifts exist by the formal smoothness of Theorem \ref{theorem:p-divisible-formally-smooth}, and the $p^n$th multiple is independent of this choice for $n$ sufficiently large by a lemma of Drinfeld \cite[Lemma 1.1.2]{katz:serre-tate}.

\subsection{The universal vector extension}

For $R$ in which $p$ is nilpotent, and $G/R$ a $p$-divisible group, we denote by $\bE(G)$ the universal vector extension of $G$,
\[ 1 \rightarrow \omega_{G^\vee} \rightarrow \bE(G) \rightarrow G \rightarrow 1. \]
There is a natural map $s_G: \wt{G} \rightarrow \bE(G)$ sending $(g_0, g_1, \ldots) \in \wt{G}(S)$ to $p^n g_n'$ for $n$ sufficiently large and $g_n'$ any lift of $g_n$ to $\bE(G)(S)$; this is well-defined since $\omega_{G^\vee}$ is annihilated by the same power of $p$ that annihilates $R$.  
\begin{remark}
From the construction of the universal vector extension in \cite{messing:crystals-associated-to-barsotti-tate}, we find that $\bE(G)$ is the push-out of the extension (\ref{equation:universal-cover-exact}) by the natural map $T_p G \rightarrow \omega_{G^\vee}$ sending $x$ to $x^* \frac{dt}{t}$ where we think of $x$ as a map from $G^\vee$ to $\wh{{\bbG_m}}$. Note that the map $T_p G \rightarrow \omega_{G^\vee}$ factors through $G[p^n]$ for $n$ sufficiently large (such that $p^n$ annihilates $R$ and thus $\omega_{G^\vee}$), so that $\bE(G)$ can be constructed as an fppf pushout (avoiding issues with fpqc sheafification in showing the pushout exists). These considerations lead to the following question: is there a natural topology suitable for constructions such as in the previous remark involving $T_pG$ and $\wt{G}$, but avoiding the set theoretic issues of the fpqc topology? 
\end{remark}

\subsubsection{Crystalline nature}
If $R' \rightarrow R$ is a nilpotent divided powers thickening, $G_0$ and $H_0$ are $p$-divisible groups over $R$, $G$ and $H$ are lifts of $G_0$ and $H_0$, respectively, to $R'$, and $\varphi: G_0 \rightarrow H_0$ is a morphism, then we obtain a morphism $\bE(\varphi)(R): \bE(G_0) \rightarrow \bE(H_0)$ by the universality of $\bE(G_0)$ (using that $\varphi^* \bE(H_0)$ is a vector extension of $G_0$). Messing \cite[Theorem  IV.2.2]{messing:crystals-associated-to-barsotti-tate} shows that there is a functorial lift
\[ \bE(G) \xrightarrow{\bE(\varphi)(R')} \bE(H). \]
By \cite[Lemma 3.2.2]{scholze-weinstein:p-div}, the following diagram commutes:
\begin{equation}\label{diagram:universal-covers-and-vector-extensions} \xymatrix{ 
\tilde{G}(R') \ar[dd]^{s_G} \ar[r]^{\sim} & \tilde{G_0}(R) \ar[d]^{s_{G_0}} \ar[r]^{\tilde{\varphi}} & \tilde{H_0}(R) \ar[d]^{s_{H_0}} \ar[r]^\sim & \tilde{H}(R') \ar[dd]^{s_H} \\ 
& \bE(G_0)(R) \ar[r]^{\bE(\varphi)(R)} & \bE(H_0)(R) & \\
\bE(G)(R') \ar[ur] \ar[rrr]^{\bE(\varphi)(R')} & & & \ar[ul] \bE(H)(R'). 
} 
\end{equation}

Passing to Lie algebras, we obtain a (nilpotent) crystal in locally free $\calO_{\crys}$-modules $\bD(G_0)$ whose value on a nilpotent divided powers thickening $R' \rightarrow R$ is $ \Lie \bE(G^\vee)$ where $G$ is any lift of $G_0$ to $R'$. This vector bundle is equipped with an integrable connection $\nabla_\crys$, and the assignment $G_0 \rightarrow (\bD(G_0)(R'), \nabla_\crys)$ is a contravariant functor: given $\varphi: G_0 \rightarrow H_0$ we obtain a map $\bD(H_0) \rightarrow \bD(G_0)$ from the construction $\bE(\varphi^\vee)$.  A specific choice of a lift $G$ gives rise to a Hodge filtration
\[ 0 \rightarrow \omega_{G} \rightarrow \Lie \bE(G^\vee) = \bD(G)(R') \rightarrow \Lie G^\vee \rightarrow 0. \]
In the remainder of this article we will usually be working with a fixed lift $G$, thus we avoid the notation $\bD$ and prefer to write $\Lie \bE(G^\vee)$ instead.

\subsection{An important example}
We now explain how to compute the maps in diagram (\ref{diagram:universal-covers-and-vector-extensions}) when $G_0=\bbQ_p/\bbZ_p$ and $H_0=\mu_{p^\infty}$. 

For $G=\bbQ_p / \bbZ_p$, $\wt{G}=\bbQ_p$. Then, $\bE(G) = \bbQ_p \times \bbG_a / \bbZ_p$, where we have identified $G^\vee$ with $\mu_{p^\infty}$ and $\omega_{G^\vee}$ with $\bbG_a$ using the basis $\frac{dt}{t}$, and 
$\bbZ_p$ is included anti-diagonally, i.e. by $z \mapsto (z,-z)$.  Here $s_{G}$ is the map $a \mapsto (a,0)$. 

For $H=\mu_{p^\infty}$, $\bE(H)=H$, and $s_{H}$ is the map 
\[ \wt{\mu_{p^\infty}} \rightarrow \mu_{p^\infty},\; (g_0, g_1, \ldots) \mapsto g_0.\] 

A map from $\QQ_p/\ZZ_p$ to $\mu_{p^\infty}$ over $R$ is an element $(g_0, g_1, \ldots) \in T_p \mu_{p^\infty}(R)$. Because $T_p \mu_{p^\infty}(R) \subset \wt{\mu_{p^\infty}}(R) \isoeq \wt{\mu_{p^\infty}}(R')$ and the latter is a $\bbQ_p$-vector space, it induces a map from $\QQ_p$ to $\wt{\mu_{p^\infty}}$. If we write $(g_0', g_1', \ldots)$ for the element of $\wt{\mu_{p^\infty}}(R')$ lifting $(g_0, g_1, \ldots)$ (i.e. the image of $1 \in \bbQ_p$), then the map 
\[ \bE(\bbQ_p/\bbZ_p)(R') \xrightarrow{\bE(\varphi)(R')} \bE(\mu_{p^\infty})(R')\]
is induced by the map
\[ \bbQ_p \times \bbG_a (R') \rightarrow \mu_{p^\infty}, (z, x) \mapsto \tilde{\varphi}(z)_0 \cdot \mathrm{exp}( x \log(g_0') ). \]
Here we have written $\tilde{\varphi}$ for the composition of the arrows at the top of the diagram (\ref{diagram:universal-covers-and-vector-extensions}) and the subscript $0$ to denote its zeroth component. The exponential and logarithm make sense because $g_0'$ is congruent to $1$ mod the kernel $I$ of $R' \rightarrow R$, which is a nilpotent divided powers ideal. Because $\mathrm{exp}( z \log(g_0') ) = (g_0')^z$ for $z \in \bbZ_p$, we find that the map is zero on the anti-diagonally embedded $\bbZ_p$. In particular, we deduce the following lemma, which we will use in our verification of Theorem \ref{maintheorem:big-hecke}.

\begin{lemma}\label{lemma:vector-extension-map}
Suppose $(g_0, g_1, \ldots)$ is an element of $\wt{\mu_{p^\infty}}(R)$ such that $g_0 \equiv 1 \mod I$ for a nilpotent divided powers ideal $I \subset R$. Then, the induced map
\[ \Lie \bE (\bbQ_p/\bbZ_p) = \bbG_a \cdot \frac{dt}{t} \rightarrow \bbG_a \cdot t \partial_t = \Lie \bE( \mu_{p^\infty})  \]
is multiplication by $\log g_0$. 
\end{lemma}

\subsection{Comparing the Gauss-Manin and crystalline connections}\label{ss:comparing-gm-and-crys}\newcommand{\GM}{\mathrm{GM}}
Let $S$ be a scheme where $p$ is locally nilpotent, let $\pi: A \rightarrow S$ be an an abelian scheme, and write $A^\vee$ for the dual abelian scheme. We have the relative de Rham cohomology 
\[ V_\dR: = R^1 \pi_* \Omega_{A/R}^\bullet \]
with Hodge filtration
\[ 0 \rightarrow \omega_{A} \rightarrow V_\dR \rightarrow \Lie A^\vee \rightarrow 0. \]
We also have the universal extension of $\bE(A[p^\infty]^\vee)=\bE(A^\vee[p^\infty])$,
\[ 1 \rightarrow \omega_{A} \rightarrow \bE(A[p^\infty]^\vee) \rightarrow A[p^\infty]^\vee \rightarrow 1 \]
and the induced Hodge filtration on $\Lie EA[p^\infty]^\vee$
\[ 0 \rightarrow \omega_{A} \rightarrow \Lie \bE(A[p^\infty]^\vee)\rightarrow \Lie A^\vee \rightarrow 0 \]
(note we have identified $\omega_{A}$ with $\omega_{A[p^\infty]}$ and $\Lie A^\vee$ with $\Lie A^\vee[p^\infty]=\Lie A[p^\infty]^\vee$ via the natural maps). 

Now, $V_\dR$ is equipped with the Gauss-Manin connection $\nabla_{\GM}$, and $\Lie A^\vee[p^\infty]$ is equipped with a connection $\nabla_{\crys}$ via the crystalline nature of the universal vector extension. The work of Mazur-Messing \cite{mazur-messing} shows
\begin{theorem}\label{theorem:vector-extension-and-gauss-manin}
There is a functorial isomorphism of filtered vector bundles with connection
\[ (\Lie \bE(A[p^\infty]^\vee), \nabla_\crys) \cong (V_\dR, \nabla_\GM). \]
inducing the identity on the associated graded bundles for the Hodge filtrations. 
\end{theorem}
\begin{proof} The identity between $\Lie \bE(A[p^\infty]^\vee)$ with its Hodge filtration as constructed above and $\Lie \mathrm{Extrig}(A, \bbG_m)$ follows from the discussion of \cite[I.2.6]{mazur-messing}. The stated isomorphism then follows from the results of \cite[II.1]{mazur-messing}; in particular, the functoriality follows from \cite[II.1.6]{mazur-messing}. 
\end{proof}

\subsubsection{Connections and vector fields}\label{sss:connections-and-vector-fields}
In preparation for our application of Theorem \ref{theorem:vector-extension-and-gauss-manin}, we now recall the relation between some different perspectives on connections. We write $D=\Spec \bbZ[\epsilon]/\epsilon^2$, the dual numbers.

Given a vector bundle with connection $(V, \nabla)$ over $S$, and a vector field $t$, viewed as a map 
\[ t: D \times S \rightarrow S, \]
we obtain an isomorphism of vector bundles on $D \times S$
\[ \nabla_t: t^* V_\dR \rightarrow 0^* V_\dR \]
where $0$ is the zero vector field. It will be useful to make this isomorphism explicit when $S=\Spec R$ and $M$ is the $R$-module of sections of $V$ over $\Spec R$. Then the map $t$ is given by 
\begin{align*}
\alpha_t:R &\rightarrow R[\epsilon] \\
r &\mapsto r + dr(t)\epsilon 
\end{align*}
and the zero section is given by
\begin{align*}
\alpha_{0}:R &\rightarrow R[\epsilon] \\
r &\mapsto r  
\end{align*}
The isomorphism $\nabla_t$ is then given in coordinates by
\begin{eqnarray}
\label{equation:connection-isomorphism}
\nabla_t: R[\epsilon] \otimes_{\alpha_t} M \rightarrow R[\epsilon] \otimes_{\alpha_0} M \\
1 \otimes m \mapsto 1 \otimes m + \epsilon \otimes \nabla_t(m). 
\end{eqnarray}
where by abuse of notation we have also written $\nabla_t$ for the derivation $M \rightarrow M$ associated to $t$ by $\nabla$.

\subsection{Serre-Tate lifting theory}\label{ss:serre-tate}
\newcommand{\Ell}{\mathrm{Ell}}
\newcommand{\Def}{\mathrm{Def}}

For $R$ a ring in which $p$ is nilpotent, and ${R_0 = R/I}$ for $I$ a nilpotent ideal, let 
\[ \Def(R,R_0) \]
be the category of triples
\[ (E_0, G, \epsilon) \]
where $E_0/R_0$ is an elliptic curve, $G$ is a $p$-divisible group, and $\epsilon: G|_{R_0} \xrightarrow{\sim} E_0[p^\infty]$ is an isomorphism. 

We denote by $\Ell(R)$ the category of elliptic curves over $R$. There is a natural functor from $\Ell(R)$ to $\Def(R,R_0)$ 
\begin{equation} \label{eqn:DefEq} E \mapsto (E_{R_0}, E[p^\infty], \epsilon_E) \end{equation}
where $\epsilon_E$ is the canonical isomorphism
\[ E[p^\infty]_{R_0} \xrightarrow{\sim} E_{R_0}[p^\infty]. \]

The following result is due to Serre-Tate, cf. \cite[Theorem 1.2.1]{katz:serre-tate}:
\begin{theorem} \label{thm:STL}
The functor \ref{eqn:DefEq} is an equivalence of categories. 
\end{theorem}

\section{Extensions of $\bbQ_p/\bbZ_p$ by $\mu_{p^\infty}$}\label{section:extensions}

In this section we study extension of $\bbQ_p/\bbZ_p$ by $\mu_{p^\infty}$. In particular, we recall a construction from \cite[8.7]{katz-mazur} of extensions which we call \emph{Kummer} $p$-divisible groups, and prove our generalization of Dwork's equation ${\tau=\log q}$ (Theorem \ref{theorem:dworkequation} below).

\subsection{The canonical trivialization}\label{subsection:the-canonical-trivialization}
Suppose given an extension of $p$-divisible groups
\[ \calE: 1 \rightarrow \mu_{p^\infty} \rightarrow G \rightarrow \bbQ_p/\bbZ_p \rightarrow 1 \]
over a scheme $S$ where $p$ is locally nilpotent. The inclusion $\mu_{p^\infty} \rightarrow G$ induces an isomorphism $\omega_G = \omega_{\mu_{p^\infty}}$, and we denote by $\omega_\can$ the image of $\frac{dt}{t}$ in $\omega_G$. The map $G \rightarrow \bbQ_p/\bbZ_p$ induces an injection
\[ \Lie \bE\l((\bbQ_p/\bbZ_p)^\vee\r) = \Lie \bE (\mu_{p^\infty}) = \Lie \mu_{p^\infty} \rightarrow \Lie \bE(G^\vee). \]
The image is the unit root filtration, which splits the Hodge filtration; we write $u_\can$ for the image of $t \partial_t \in \Lie {\mu_{p^\infty}}$ in $\Lie \bE(G^\vee)$.  

We thus obtain a trivialization 
\begin{equation}\label{equation:extension-vb-canonical-basis} \Lie \bE(G^\vee) = \bbG_a \cdot \omega_\can \times \bbG_a \cdot u_\can \end{equation}
where the first term spans the Hodge filtration and the second the unit root filtration. The elements $t \partial_t$ and $\frac{dt}{t}$ are flat for the connections on $\Lie \bE\l((\bbQ_p/\bbZ_p)^\vee\r)$ and $\Lie \bE(\mu_{p^\infty}^\vee)$, respectively, and thus we find that in the basis (\ref{equation:extension-vb-canonical-basis}), $\nabla_\crys$ is lower nilpotent, i.e.
\[ \nabla_\crys(\omega_\can) \in u_\can \cdot \Omega_S, \; \nabla_\crys(t \partial_t) = 0. \]
In particular, the extension $\calE$ determines a differential form
\[ d\tau_{\calE} := \frac{ \nabla_\crys(\omega_\can)}{u_\can} \in \Omega_S. \]
The notation is a slight abuse, as in general there is no function $\tau_{\calE}$ in $\calO(S)$ whose differential is equal to $d\tau_{\calE}$; nevertheless, as we will see below, it is natural to think of this as the differential of Dwork's divided powers coordinate $\tau$. 
  
\subsection{Kummer $p$-divisible groups}\label{subsection:kummer-p-divisible-groups}

For $R$ a ring and $q \in R^\times$, we
will construct an extension of $p$-divisible groups over $\Spec R$, 
\[ \calE_q: 1 \rightarrow \mu_{p^\infty} \rightarrow G_q \rightarrow \bbQ_p/ \bbZ_p \rightarrow 1. \]
We call the extensions $\calE_q$ arising from this construction \emph{Kummer $p$-divisible groups} (for reasons explained below in Remark \ref{remark:kummer-name-explanation}). This construction is due to Katz-Mazur \cite[8.7]{katz-mazur} (who work in the univeral case over $\bbZ[q,q^{-1}]$), but because it will be useful later we give the details and some complements below. 

\newcommand{\Roots}{\mathrm{Roots}}
We first consider the fppf sheaf in groups
\[ \Roots_q \subset \GG_m \times \ZZ[1/p] \]
consisting of pairs $(x, m)$ such that for $k$ sufficiently large, $x^{p^k} = q^{p^k m}$.

Projection to the second component gives a natural map $\Roots_q \rightarrow \bbZ[1/p]$. The kernel is identified with $\mu_{p^\infty}$, and the projection admits a canonical section over $\bbZ$ by $1 \mapsto (q,1)$. We consider the quotient by the image of this section
\[ G_q:=\Roots_q / \ZZ. \]

\begin{lemma} $G_q$ is a $p$-divisible group, and the maps 
\[ \mu_{p^\infty} \rightarrow \Roots_q \textrm{ and } \Roots_q \rightarrow \bbZ[1/p] \]
induce the structure of an extension
\[ \calE_q: 1\rightarrow \mu_{p^\infty} \rightarrow G_q[p^\infty] \rightarrow \bbQ_p/ \bbZ_p \rightarrow 1 \]
\end{lemma}
\begin{proof}
If we let $\Roots_q'$ be the subsheaf of $\Roots_q$ of elements $(x,m)$ with $m \in \bbZ[1/p], 0 \leq m <1$, then the group law induces an isomorphism 
\[ \Roots_q' \times \bbZ \rightarrow \Roots_q.\]
Thus, $\Roots_q'$ as a sheaf of sets is isomorphic to $\Roots_q / \bbZ$, and for $A$ an $R$-algebra with $\Spec A$ connected,
\[ G_q(A)=\Roots_q(A) / (q,1)^{\bbZ} \]
and any element of $G_q(A)$ has a unique representative of the form $(x, m) \in \Roots_q(A)$ with $0 \leq m < 1$. Such an element is $p^k$-torsion if and only if $m \in 1/p^k \bbZ$ and $x^{p^k}=q^{p^k m}$. In particular, we find that $G_q = \colim G_q[p^k]$. Moreover, multiplication by $p$ is an epimorphism because taking a $p$th root of $x$ gives an fppf cover. Thus, to see that $G_q$ is a $p$-divisible group, it remains only to see that $G_q[p]$ is a finite flat group scheme. In fact, for any $k$, our description of elements shows that $G_q[p^k]$ is represented by
\[  \sqcup_{0\leq a \leq p^{k}-1} \Spec  R[x]/(x^{p^k} - q^a), \]
with multiplication given by ``carrying," i.e. for $x_1$ a root of $q^{a_1}$ and $x_2$ a root of $q^{a_2}$, in the group structure
\[ x_1 \cdot x_2 = \begin{cases} x_1 x_2  &\textrm{ as a root of $q^{a_1+a_2}$ if $a_1 + a_2 < p^k$} \\
x_1 x_2 / q & \textrm{ as a root of $q^{a_1+a_2 - p^k}$ if $a_1 + a_2 \geq p^k$}. \end{cases} \]
This is a finite flat group scheme.

Finally, the extension structure is clear from definition. 
\end{proof}

\begin{remark}\label{remark:pairing}
Let $\Roots_{q,k} \subset \Roots_q$ be the elements $(x,m)$ such that $p^k m \in \bbZ$ and $x^{p^k}=q^{p^k m}$, so that $\Roots_{q,k}/\bbZ = G_q[p^k].$ We have a natural pairing
\[ \Roots_{q,k} \times \Roots_{q^{-1},k} \rightarrow \mu_{p^k} \]
given by $\langle (g,a),\, (h,b) \rangle = g^{p^k b} h^{p^k a}$, which induces a perfect pairing
\[ G_q[p^k] \times G_{q^{-1}}[p^k] \rightarrow \mu_{p^k}. \]
It realizes an isomorphism of extensions
\[ \calE_q^\vee \xrightarrow{\sim} \calE_{q^{-1}} \]
Note that at the level of groups $G_q \isoeq G_{q^{-1}}$; the extension structures $\calE_{q}$ and $\calE_{q^{-1}}$  differ by composition with an inverse on either $\bbQ_p/\bbZ_p$ or $\mu_{p^\infty}$.   
\end{remark}

\begin{example}\label{example:kummer-groups} The following three examples will be useful later on:
\begin{enumerate}
\item For the Tate curve $\Tate(q)$ over $\ZZ((q))$, $\Tate(q)[p^\infty]=G_q[p^\infty]$ (\cite[8.8]{katz-mazur}). 
\item For $A$ an Artin local ring with perfect residue field $k$ of characteristic $p$, any lift of the trivial extension $\mu_{p^\infty} \times \QQ_p/\ZZ_p$ over $k$ to $A$ is uniquely isomorphic to $\calE_q$ for a unique $q \in \widehat{\GG}_m(A)$, and $q^{-1}$ is the Serre-Tate coordinate of the lift (cf. Remark \ref{remark:alternate-construction} below). 
\item 
The formation of $\calE_q$ commutes with base change. In particular, there is a universal Kummer $p$-divisible group,
\[ \calE_{q_\univ} / \bbG_{m,\bbZ} = \Spec \bbZ[q_\univ^{\pm 1}], \]
so that for any $q \in R^\times$, $\calE_q / \Spec R$ is given via pullback of $\calE_{q_\univ}$ through the map $\Spec R \rightarrow \bbG_m$ given by $q \in R^\times = \bbG_m(R)$. 

\end{enumerate}
\end{example}

\begin{remark}\label{remark:kummer-extensions-funny}
Over a general $R$, not every extension of $\bbQ_p/\bbZ_p$ by $\mu_{p^\infty}$ is a Kummer $p$-divisible group, and for those which are, there may not be a canonical choice of $q$ as in the Artin local case. In particular, the extension given by the $p$-divisible group of the universal trivialized elliptic curve over $\VV_{\Katz, \bbF_p}$ is not a Kummer $p$-divisible group, as we explain in \ref{subsec:not-kummer} below. 
\end{remark}

\begin{remark}\label{remark:kummer-name-explanation}
For any $k \geq 0$, consider the Kummer sequence
\[ 1 \rightarrow \mu_{p^k} \rightarrow \bbG_m \xrightarrow{x \mapsto x^{p^k}} \bbG_m \rightarrow 1 \]
We may take the pull-back by 
\[ \ZZ \mapsto \bbG_m, \; 1 \mapsto q \]
to obtain an extension
\[ 1 \rightarrow \mu_{p^k} \rightarrow p^k-\Roots_{q} \rightarrow \bbZ \rightarrow 1. \]
Equivalently, this extension is the image of $q$ under the coboundary map
\[ \bbG_m(R) \rightarrow H^1_{\fppf}(\Spec R, \mu_{p^k})=\Ext^1(\bbZ, \mu_{p^k}). \]
There is a natural map 
\[ p^k-\Roots_{q} \rightarrow \Roots_q. \]
Indeed, an element of $p^k-\Roots_{q}$ is a pair $(x, a) \in \bbG_m \times \bbZ$ such that $x^{p^k}=q^a$, and this is mapped to the pair
\[ (x, a/p^k) \in \bbG_m \times \bbZ[1/p] \]
which lies in $\Roots_q$. This is an isomorphism of $p^k-\Roots_{q}$ onto its image, which consists of all $(x,m)$ such that $m \in \frac{1}{p^k} \bbZ$ and $x^{p^k} = q^{p^k m}$ -- this is what we denoted by $\Roots_{q,k}$ in Remark \ref{remark:pairing}. In particular, the map $\Roots_q \rightarrow G_q$ induces an isomorphism
\[ p^k-\Roots_q / (q,p^k)^\bbZ \xrightarrow{\sim} G_q[p^k]. \]   
It is for this reason that we refer to $\calE_q$ as a Kummer $p$-divisible group. 

Note that there are also natural maps between the Kummer sequences as $k$ varies inducing the obvious inclusions as sub-functors of $\Roots_q$, and we find 
\[ \Roots_q = \colim_k p^k-\Roots_{q}. \]
To construct $G_q$ we can also take the colimit already at the level of the Kummer sequences. If we do so, we obtain the (exact) exponential sequence 
\begin{equation*}\label{eqn:exponential}
\calE_{\mathrm{exp}}: \; 1 \rightarrow
\mu_{p^\infty}\rightarrow \bbG_m \rightarrow \colim \left( \bbG_m \xrightarrow{x \mapsto x^p} \bbG_m \xrightarrow{x \mapsto x^p} \ldots \right) \rightarrow 1.  
\end{equation*}
There is a map 
\[ \alpha: \ZZ \rightarrow \bbG_m \]
sending $1$ to $q$ which extends uniquely to a map
\[ \alpha_{1/p}: \ZZ[1/p] \rightarrow \colim \left( \bbG_m \xrightarrow{x \mapsto x^p} \bbG_m \xrightarrow{x \mapsto x^p} \ldots \right). \] 
Then, essentially by definition, $\alpha_{1/p}^* \calE_{\mathrm{exp}}$ is the extension
\[ 1 \rightarrow \mu_{p^\infty} \rightarrow \Roots_q \rightarrow \bbZ[1/p] \rightarrow 1. \]
The map $\alpha \times \Id: \bbZ \rightarrow \bbG_m \times \bbZ[1/p]$ factors through $\alpha_{1/p}^* \calE_{\mathrm{exp}}$ and we find 
\[ \calE_q = \alpha_{1/p}^* \calE_{\mathrm{exp}} / \alpha \times \Id (\bbZ). \]

\end{remark}

\begin{remark}\label{remark:alternate-construction}
In this remark we explain a third construction of $G_q$ and the connection to Serre-Tate coordinates:  
Consider the extension 
\begin{equation}\label{eqn:integers-p-extension} 1 \rightarrow \ZZ \rightarrow \ZZ[1/p] \rightarrow \QQ_p/\ZZ_p \rightarrow 1\end{equation}
We obtain an extension of $\QQ_p/\ZZ_p$ by $\GG_m$, $A_q$, as the push-out of (\ref{eqn:integers-p-extension}) by 
\begin{equation}\label{eqn:pushout-map}
\ZZ \rightarrow \GG_m, \;\;
1 \mapsto q^{-1}. 
\end{equation}
We claim there is a natural isomorphism
$G_q \isoeq A_q[p^\infty]$
respecting the extension structure. 
To see this, note that the push-out $A_q$ is constructed as the quotient of $\GG_m \times \ZZ[1/p]$ by the subgroup generated by $(q,1)$. Then, the $p^\infty$-torsion is just the image of $\Roots_q$ in $A_q$, as desired. 

We note that if $q \in \widehat{\bbG_m}(R)$, then taking the push-out and passing to $p^\infty$ torsion is equivalent to just taking the pushout under (\ref{eqn:pushout-map}) viewed as a map to $\widehat{\bbG_m}$. Thus, when restricted to $q \in \widehat{\bbG_m}(R)$ for Artin local $R$ with perfect residue field, our construction gives the extension of $\bbQ_p/\bbZ_p$ by $\mu_{p^\infty}$ with Serre-Tate coordinate $q^{-1}$ (cf. \cite[Appendix 2.4-2.5]{messing:crystals-associated-to-barsotti-tate}). 
\end{remark}

We will need the following result on maps between Kummer $p$-divisible groups:

\begin{lemma}\label{lemma:maps-between-kummer} Isomorphisms  $\calE_q \xrightarrow{\sim} \calE_{q'}$ are identified with the fiber above $q'/q$ for the map
\[ \widetilde{\GG_m} \rightarrow \GG_m \]
sending $(x_0, x_1, \ldots)$ to $x_0$. 
\end{lemma}
\begin{proof}
Let $t=q'/q$. Suppose given a compatible system of roots $t^{1/p^n}$ of $t$. We obtain an isomorphism between $G_q[p^n]$ and $G_q'[p^n]$ respecting the extension structure by sending an element $(a, k/p^n)$ to $(a t^{k/p^n}, k/p^n)$, and these are compatible for varying $n$. 

Conversely, given an isomorphism $\psi:G_q[p^\infty] \rightarrow G_{q'}[p^\infty]$ compatible with the extension structures, if we restrict to $\psi_n:G_q[p^n]\rightarrow G_{q'}[p^n]$, then for any $(a,1/p^n)\in G_q[p^n]$, $\psi (a,1/p^n) = (a', 1/p^n)$ for $a'$ such that $a'^{p^n}=q'$, and $a'/a$ is $p^n$th root of $t$ that is \emph{independent} of $a$ because two choices of $a$ differ by an element of $\mu_{p^n}$; it thus comes from an element of $R^\times$, and the isomorphism at level $p^n$ is as above; the roots of $t$ chosen by varying the level then must also be compatible, giving an element of $\tilde{\GG_m}$ mapping to $t$. 
\end{proof}

\subsection{Dwork's equation $\tau=\log q$}
The universal deformation of $\mu_{p^\infty} \times \bbQ_p/\bbZ_p$ over $\overline{\bbF_p}$, $G_{\univ}/\Spf W(\overline{\bbF}_p)[[t]]$, is canonically an extension 
\[\calE: 1 \rightarrow \mu_{p^\infty} \rightarrow{G_{\univ}} \rightarrow \bbQ_p/\bbZ_p \rightarrow 1. \]
Because $W(\overline{\bbF}_p)[[t]]$ is pro-Artin local, $\calE = \calE_q$, for a unique $q \equiv 1 \mod (t,p)$, and $q^{-1}$ is the Serre-Tate coordinate (cf. Example \ref{example:kummer-groups}-(2) and Remark \ref{remark:alternate-construction}). The $W(\overline{\bbF_p})$ point $x_\can$ with $q=1$ parameterizes the unique split lift to $W(\overline{\bbF}_p)$, the canonical lifting, and we can extend the canonical basis $\omega_\can|_{x_\can}, u_\can|_{x_\can}$ of $\bE(G)|_{x_\can}$ at this point to a flat basis over the divided powers envelope of $x_\can$ (the extension of $u_\can|_{x_\can}$ is just $u_\can$ itself, but $\omega_\can$ is not flat so the flat extension of $\omega_\can|_{x_\can}$ is not equal to $\omega_\can$). The position of the Hodge filtration with respect to this basis then defines a divided powers function $\tau$, and a conjecture of Dwork proven by Katz \cite{katz:serre-tate} states\footnote{Recall from Remark \ref{remark:alternate-construction} that the Serre-Tate coordinate of $\calE_q$ is $q^{-1}$!}
\[ \tau = \log q^{-1}. \]

As observed by Katz \cite{katz:l-via-moduli}, this is equivalent to computing, in the language of \ref{subsection:the-canonical-trivialization}, 
\[ d\tau_{\calE_q}=d\log q^{-1}. \]
We now give a simple proof of this result by using a very ramified base-change to split $\calE_q$. The result is valid for any Kummer $p$-divisible group:

\begin{theorem}\label{theorem:dworkequation}
For $S$ a scheme on which $p$ is locally nilpotent and 
\[ q\in \bbG_m(S) = \calO(S)^\times, \]
we have 
\[ d\tau_{\calE_q} = - d\log q = d \log q^{-1} = -\frac{dq}{q}. \]
\end{theorem}
\begin{proof}
By reduction to the universal case, it suffices to prove this for $\calE_q$ over 
\[ S=\bbG_{m, \bbZ/p^n\bbZ}=\Spec \bbZ/p^n\bbZ [q^{\pm 1}]. \]
In this case, $\Omega_S$ is free with basis $d\log q=\frac{dq}{q}$, thus it suffices to show that 
\[ \nabla_{\crys, q \partial_q}(d\tau_{\calE_q})=-1. \]
The vector field $q \partial q$, thought of as a map
\[ t: D\times S \rightarrow S \]
is given by the map of rings
\[ R \rightarrow R[\epsilon]/\epsilon^2, \, q \mapsto (1+\epsilon)q,  \]
and we can compute the isomorphism 
\[ t^* \Lie \bE(G_q^\vee) \rightarrow  0^* \Lie \bE(G_q^\vee) \]
induced by $\nabla_\crys$ as follows:

First, we observe that $t^* \calE_{q} = \calE_{(1+\epsilon)q}$ and ${0^* \calE_{q} = \calE_{q}}$, where $q$ is thought of an element of $R[\epsilon]$, and under these identifications the isomorphism 
\[ 0^* \calE_{q} \mod \epsilon \xrightarrow{\sim} t^*\calE_{q} \mod \epsilon \]
is identified with the canonical isomorphism 
\begin{equation} \label{equation:canonical-reduction-kummer-iso} \calE_{q} \mod \epsilon = \calE_{(1+\epsilon)q} \mod \epsilon  \end{equation}
given by $(1+\epsilon) q=q \mod \epsilon$.

Thus, using the description of \ref{sss:connections-and-vector-fields}, it suffices to show that the induced map 
\begin{equation}\label{equation:induced-map} \nabla_{\crys, q\partial_q}: \Lie \bE(G_{(1+\epsilon)q}^\vee) \rightarrow  \Lie \bE(G_{q}^\vee) \end{equation}
is given in the canonical bases by
\begin{equation}
\label{equation:desired-matrix} \begin{pmatrix} 1 & 0 \\ 1-\epsilon & 1 \end{pmatrix}. \end{equation}
It suffices to verify this after flat base change, so we may adjoin roots $q^{1/p^\infty}$ and $(1+\epsilon)^{1/p^\infty}$ to obtain a ring $R_\infty / (R[\epsilon]/\epsilon^2)$. 

Over $R_\infty$, the maps $1/p^n \rightarrow q^{1/p^n}$ and $1/p^n \rightarrow (1+\epsilon)^{1/p^n}q^{1/p^n}$ split $\calE_q$ and $\calE_{(1+\epsilon)q}$. In these trivializations, the canonical isomorphism (\ref{equation:canonical-reduction-kummer-iso}) is identified with the map 
\[ \mu_{p^\infty} \times \bbQ_p/\bbZ_p \rightarrow \mu_{p^\infty} \times \bbQ_p/\bbZ_p \]
given by
\[ \begin{pmatrix} 1 & ( (1+\epsilon)^{-1}, (1+\epsilon)^{-1/p}, \ldots) ) \\ 0 & 1 \end{pmatrix} \mod \epsilon. \]
The transpose map
\[ G_{(1+\epsilon) q}^\vee \rightarrow G_q^\vee \]
is identified with 
\[ \begin{pmatrix} 1 & 0 \\
 ((1+\epsilon)^{-1}, (1+\epsilon)^{-1/p}, \ldots ) & 1 \end{pmatrix} \mod \epsilon, \]
and using Theorem \ref{theorem:computation-of-action}, we concluded that over $R_\infty$, in the canonical bases the map (\ref{equation:induced-map}) is given by (\ref{equation:desired-matrix}), as desired.
\end{proof}

\section{Moduli problems for ordinary elliptic curves}\label{section:moduli-problems}
In this section, we discuss various moduli problems for ordinary elliptic curves over a base $S$ where $p$ is locally nilpotent. 

\subsection{Level structures}

\subsubsection{Prime-to-$p$ level structure} 
\newcommand{\Sch}{\mathrm{Sch}}
\newcommand{\Zar}{\mathrm{Zar}}
For $T$ a topological space, we write $\underline{T}$ for the functor on $\Sch$ sending $S$ to $\Cont(|S|, T)$, where $|S|$ denotes the topological space underlying $S$. 

Given an elliptic curve $E/S$ over a scheme $S$, we define the prime-to-$p$ Tate module 
\[ T_{\wh{\bbZ}^{(p)}} E:= \lim_{(n,p)=1} E[n], \]
as a functor on $\Sch/ S$, where the transition map from $E[n']$ to $E[n]$ for $n|n'$ is multiplication by $n' / n$. The transition maps are affine, so the prime-to-$p$ Tate module is representable. We define the adelic prime-to-$p$ Tate module as the sheaf on $S_\Zar$
\[ V_{\AA_f^{(p)}} E := T_{\wh{\bbZ}^{(p)}} E \otimes_\bbZ \QQ.  \]
The prime-to-$p$ Tate module is functorial for quasi-$p$-isogenies, and the prime-to-$p$ adelic Tate module is functorial for quasi-isogenies. 

An integral prime-to-$p$ infinite level structure on $E$ is a trivialization 
\[ \varphi_{\wh{\bbZ}^{(p)}}: T_{\wh{\bbZ}^{(p)}} E \xrightarrow{\sim} \underline{\left(\wh{\bbZ}^{(p)}\right)^2}. \]

An rational prime-to-$p$ infinite level structure on $E$ is a trivialization
\[ \varphi_{\bbA_f^{(p)}}:  V_{\AA_f^{(p)}} E \xrightarrow{\sim} \underline{\left(\bbA_f^{(p)}\right)^2}. \]
The degree of a rational prime-to-$p$ infinite level structure is the index 
\[ \left[ \varphi_{\bbA_f^{(p)}}\l(T_{\wh{\bbZ}^{(p)}} E\r) :  \left(\wh{\bbZ}^{(p)}\right)^2 \right] .\]

\subsubsection{Structures at $p$}
If $R$ is a ring in which $p$ is nilpotent, and $E/\Spec R$ is an elliptic curve, we will consider the following presheaves on $\Nilp_R^\op$:
\begin{enumerate}
\item The $p$-divisible group 
\[ E[p^\infty] := \colim E[p^n] \]
\item The formal group $\wh{E}=E[p^\infty]^\circ$ (as defined already in \ref{subsection:formal-neighborhoods-and-lie-algebras}),
\item The universal cover of $E[p^\infty]$,
\[ \wt{E[p^\infty]} := \lim_{p} E[p^\infty] \]
(as defined already in \ref{subsection:universal-covers}.)
\end{enumerate}
The formal group and $p$-divisible group of $E$ are functorial with respect to quasi-prime-to-$p$-isogenies of $E$, and the universal cover of $E[p^\infty]$ is functorial with respect to quasi-isogenies of $E$.

\subsubsection{Katz level structure at $p$}

A Katz level structure on $E[p^\infty]$ is a trivialization
\[ \wh{\varphi}_p: E[p^\infty]^\circ \xrightarrow{\sim} \mu_{p^\infty}. \]

\subsubsection{Infinite level structure at $p$}

An integral ordinary infinite level structure on $E[p^\infty]$ is a trivialization 
\[ \varphi_p: E[p^\infty] \xrightarrow{\sim} \mu_{p^\infty} \times \bbQ_p/\bbZ_p. \]
A rational ordinary infinite level structure on $E[p^\infty]$ is a trivialization
\[ \varphi_p: \wt{E[p^\infty]} \xrightarrow{\sim} \wt{\mu_{p^\infty}} \times \bbQ_p \; \left( = \wt{ \mu_{p^\infty} \times \bbQ_p/\bbZ_p} \right). \]
The degree of a rational ordinary infinite level structure is the degree of the corresponding quasi-isogeny $E[p^\infty] \rightarrow \mu_{p^\infty} \times \bbQ_p/\bbZ_p$.

\subsection{Polarization and the Weil pairing}
Our moduli problems will need to take into account a polarization, so we first recall some notation. For $R$ a ring in which $p$ is nilpotent and $E/\Spec R$ an elliptic curve, the $p^n$-Weil pairing is a perfect antisymmetric pairing
\[ e_{p^n, E}: E[p^n] \times E[p^n] \rightarrow \mu_{p^n}. \]
It induces an anti-symmetric $\bbQ_p$-bilinear pairing
\[ \wt{e}_E: \wt{E[p^\infty]} \times \wt{E[p^\infty]} \rightarrow \wt{\mu_{p^\infty}} \]
given by 
\[ \wt{e}_E ( (a_k), (b_k) ) = (c_k) \]
where
\[ c_k = \left( e_{p^t, E}(a_i, b_j) \right)^{p^s} \]
for $i+j= s + t + k$ and $t$ large enough that $a_i, b_j \in E[p^t]$ so that the right-hand side is defined. 

\begin{lemma}\label{lemma:isogeny-action-weil-pairing}
If $f: E\rightarrow E'$ is an isogeny or quasi-prime-to-$p$ isogeny, then
\[ f^* e_{p^n, E'} = e_{p^n, E}^{\deg f}. \]
If $f$ is a quasi-isogeny,
\[ f^* \wt{e}_{E'} = \wt{e}_E^{\deg f}. \]
\end{lemma}
\begin{proof}
The first equation for isogenies is a well-known property of the Weil pairing, and the second equation for isogenies is then immediate from the definition of $\wt{e}$. Once the isogeny statements are established, the quasi-isogeny statements follow as raising to a prime-to-$p$ integer power is invertible on $\mu_{p^n}$ and raising to any integer power is invertible on $\wt{\mu_{p^\infty}}$. 
\end{proof}

In particular, we note that the $p^n$ Weil pairings $e_{p^n}$ are functorial in degree one quasi-prime-to-$p$-isogenies of $E$, and the universal cover Weil pairing $\wt{e}$ is functorial in degree one quasi-isogenies of $E$. 

Below we will also consider the standard pairing 
\begin{equation}\label{eqn:standard-pairing} \langle, \rangle_{\std}: (\wt{\mu_{p^\infty}} \times \bbQ_p)^2 \rightarrow \wt{\mu_{p^\infty}}, \; \langle (x_1, x_2) , (y_1,y_2) \rangle_\std = x_1^{y_2} y_1^{-x_2}. \end{equation}

\subsection{The Igusa moduli problem of Caraiani-Scholze}
The Igusa moduli problem of Caraiani-Scholze \cite{caraiani-scholze:generic} classifies, for $\Spec R \in \Nilp_{\bbZ_p}^{\op}$ the set of triples
\[ (E, \varphi_p, \varphi_{\bbA_f^{(p)}}) \]
where $E/R$ is an elliptic curve up to isogeny, $\varphi_p$ is a rational ordinary infinite level structure on $E[p^\infty]$, and $\varphi_{\bbA_f^{(p)}}$ is a rational infinite\footnote{Caraiani-Scholze work with finite prime-to-$p$ level structure but passing to infinite level away from $p$ poses no serious difficulties.} prime-to-$p$ level structure. 

\subsubsection{Representability} By work of Caraiani-Scholze \cite{caraiani-scholze:generic}, this moduli problem is represented by an affine $p$-adic formal scheme over $\bbZ_p$, 
\[ \Ig_\CS = \Spf \bbV_{\CS}. \]
The ring $\bbV_{\CS}$ is flat over $\bbZ_p$; indeed, it is the Witt vectors of a perfect ring over $\bbF_p$. For the finite level variant this follows from  \cite[p.718]{caraiani-scholze:generic} (who work over $W(\overline{\bbF_p})$ instead of $\bbZ_p$, but this is not necessary here). Their argument applies equally well to infinite level at $p$ by taking the Witt lift of the perfect ring representing the corresponding mod $p$ moduli problem, which is just the colimit of the perfect rings representing the finite level moduli problems (equivalently, one takes the colimit of the Witt lifts of these and then $p$-adically completes). 

\begin{remark}
We will explain the construction of $\bbV_{\CS}$ in more detail from a classical perspective below. 
\end{remark}

\subsubsection{A polarized variant}
For our purposes, we will also need the polarized variant of this moduli problem. To state it, we first observe that any triple $(E, \varphi_p, \varphi_{\bbA_f^{(p)}})$ as above, we can choose a representative for the isogeny class of $E$ such that $\varphi_p$ and $\varphi_{\bbA_f^{(p)}}$ are both degree one, and that such a representative is determined up to degree one quasi-isogeny. Because the Weil pairing of an elliptic curve is preserved under degree one quasi-isogeny, we obtain a well-defined Weil pairing $\wt{e}_{(E, \varphi_p, \varphi_{\bbA_f^{(p)}})}$ on $\wt{E[p^\infty]}.$ The polarized moduli problem then parameterizes triples as above where we additionally require that
\[ (\varphi_p^{-1})^* \wt{e}_{(E, \varphi_p, \varphi_{\bbA_f^{(p)}})} = \langle, \rangle_\std \]
where $\langle,\rangle_\std$ is the pairing defined in (\ref{eqn:standard-pairing}). The polarized moduli problem is represented by a closed formal subscheme $\Ig_{\CS}^1 = \Spf \bbV_{\CS}^1 \subset \Ig_{\CS}$. 

\subsubsection{A $p$-integral version}\label{sss:p-integral-cs}
The corresponding $p$-integral moduli problem where $E$ is an elliptic curve up to quasi-prime-to-$p$-isogeny and $\varphi_p$ is an integral ordinary infinite level structure is equivalent to the rational moduli problem via the natural inclusion, and thus is also represented by $\Ig_{\CS}$ (cf., e.g., \cite[Lemma 4.3.10]{caraiani-scholze:generic}). Similarly, there is an equivalent $p$-integral formulation of the polarized moduli problem.  

\subsection{The Igusa moduli problem of Katz}
The Igusa moduli problem of Katz  classifies, for $\Spec R \in \Nilp_{\bbZ_p}^{\op}$ the set of triples
\[ (E, \wh{\varphi}_p, \varphi_{\bbA_f^{(p)}}) \]
where $E/R$ is an elliptic curve up to quasi-prime-to-$p$-isogeny, 
\[ \wh{\varphi}_p :  E[p^\infty]^\circ \xrightarrow{\sim}  \mu_{p^\infty} \;\; (\textrm{equivalently}, \wh{\varphi}_p: \wh{E} \xrightarrow{\sim}  \wh{\bbG_m})  \]
is a Katz level structure, and $\varphi_{\bbA_f^{(p)}}$ is a rational infinite\footnote{Katz works with specific finite prime-to-$p$ level structures, but passing to infinite level away from $p$ poses no serious difficulties.} prime-to-$p$ level structure. 

\subsubsection{Representatiblity}
By work of Katz \cite{katz:l-via-moduli} this moduli problem is represented by an affine $p$-adic formal scheme over $\bbZ_p$
\[ \Ig_{\Katz} = \Spf \bbV_{\Katz} \]
The ring $\bbV_{\Katz}$ is flat over $\bbZ_p$. It is the $p$-adic completion of the colimit of $p$-adically complete rings representing the moduli problem parameterizing arithmetic $\Gamma_1(p^n)$ structure at $p$ and finite level structure away from $p$ (over all $n$ and all finite levels). 

\subsubsection{A polarized lift}\label{sss:polarized-lift-Katz}
Given a triple  $(E, \wh{\varphi}_p, \varphi_{\bbA_f^{(p)}})$ as above, we may choose a representative for $E$ such that $\varphi_{\bbA_f^{(p)}}$ is of degree one, and such a triple is unique up to degree one quasi-prime-to-$p$-isogeny of $E$. Thus we obtain for each $n$ a well-defined Weil pairing $e_{(E, \varphi_{\bbA_f^{(p)}}), n}$ on $E[p^n]$, and this induces a pairing
\[ (,)_{(E, \varphi_{\bbA_f^{(p)}}), n}: E[p^n]^\circ \times E[p^n]/E[p^n]^\circ \rightarrow \mu_{p^n}. \]
Using this pairing, $\wh{\varphi}_p$ induces an isomorphism
\[ \varphi_p^\et: E[p^\infty]/E[p^\infty]^\circ \xrightarrow{\sim} \bbQ_p/\bbZ_p \]
uniquely determined by the condition that, for each $n$, 
\[ \left( \left(\wh{\varphi}_p\right)^{-1} (\bullet) , \left(\varphi_p^\et\right)^{-1}(1/p^n) \right)_{(E, \varphi_{\bbA_f^{(p)}}),n} = \Id_{\mu_{p^n}}. \]

We can rephrase this by saying that $\Ig_{\Katz}$ also represents the moduli problem classifying quadruples
\[ (E, \wh{\varphi}_p, \varphi_p^\et, \varphi_{\bbA_f^{(p)}}) \]
where $E$ and $\varphi_{\bbA_f^{(p)}}$ are as above, 
\[ \wh{\varphi}_p: E[p^\infty]^\circ \xrightarrow{\sim} \mu_{p^\infty} \textrm{ and }  \varphi_p^\et: E[p^\infty]/E[p^\infty]^\circ \xrightarrow{\sim} \bbQ_p/\bbZ_p \]
are isomorphisms, and the pairing 
\[ \wt{\mu_{p^\infty}} \times \bbQ_p \rightarrow \wt{\mu_{p^\infty}} \]
induced by $\tilde{e}_{(E, \varphi_{\bbA_f^{(p)}})}$, $\wh{\varphi}_p$, and $\varphi^\et_p$ is given by
\[ (a, b) \mapsto a^b.\] 

\begin{remark}\label{remark:igusa-data-extension} We note that to give the data $\wh{\varphi}_p$ and $\varphi_p^\et$ is equivalent to equipping $E[p^\infty]$ with the structure of an extension
\begin{equation}\label{equation:igusa-data-extension} \calE_{E[p^\infty], \wh{\varphi}_p, \varphi_p^\et}: 1 \rightarrow \mu_{p^\infty} \rightarrow E[p^\infty] \rightarrow \bbQ_p/\bbZ_p \rightarrow 1. 
\end{equation}
\end{remark}

\subsection{Group actions}
\renewcommand{\l}{\left}
\renewcommand{\r}{\right}
\newcommand{\ul}[1]{\underline{#1}}
\subsubsection{Automorphism groups at $p$}
We consider the \emph{twisted Borel} $B_p$, the presheaf on $\Nilp_{\bbZ_p}^\op$ defined by  
\[ B_p(R)  := \Aut\l( \l(\wt{\mu_{p^\infty}} \times \bbQ_p\r)_R \r). \]
Because there are no non-zero maps from $\wt{\mu_{p^\infty}}$ to $\bbQ_p$ and $\Hom(\bbQ_p, \wt{\mu_{p^\infty}})=\wt{\mu_{p^\infty}}$, we can write this as the matrix group
\[ B_p = \begin{pmatrix} \ul{\bbQ_p^\times} & \wt{\mu_p^\infty} \\ 0 & \ul{\bbQ_p^\times} \end{pmatrix} \]
We write $M_p$ for the diagonal subgroup, and $U_p=\wt{\mu_{p^\infty}}$ for the unipotent subgroup. We write $\det: B_p \rightarrow \ul{\bbQ_p^\times}$ for the product of the diagonal entries. 

We also consider the integral variants
\begin{align*}
 B_p^\circ & := \Aut\l(\mu_{p^\infty} \times \bbQ_p/\bbZ_p \r) = \begin{pmatrix} \ul{\bbZ_p^\times} & T_p \mu_{p^\infty} \\ 0 & \ul{\bbZ_p^\times} \end{pmatrix} \subset B_p\\
 M_p^\circ & := \Aut(\mu_{p^\infty}) \times \Aut(\bbQ_p/\bbZ_p) \subset  M_p \\
 U_p^\circ &:= \Hom(\bbQ_p/\bbZ_p, \mu_{p^\infty}) = T_p \mu_{p^\infty} \subset U_p.
\end{align*}

\subsubsection{Moduli action on $\Ig_{\CS}$ and $\Ig_{\CS}^1$}
We write $G^{(p)} = \underline{\GL_2(\bbA_f^{(p)})}$. Composition with the level structures gives an action of $B_p \times G^{(p)}$ on $\Ig_{\CS}$.

We would like to understand the subgroup preserving $\Ig_{\CS}^1$: First, it is a straightforward computation to check that the Weil pairing constructed above transforms via the character 
\[ \det_\ur: B_p \times G^{(p)} \rightarrow \ul{\bbQ_p^\times}, (b_p) \times (g_\ell)_{\ell \neq p} \mapsto  |\det(b_p)|_p \cdot \prod_{\ell \neq p} |\det(g_\ell)|_\ell. \]
On the other hand, the standard pairing $\langle, \rangle_\std$ on $\wt{\mu_{p^\infty}} \times \bbQ_p$ transforms via
\[ \det_p: B_p \times G^{(p)} \rightarrow \ul{\bbQ_p^\times}, (b_p) \times (g_\ell)_{\ell \neq p} \mapsto  \det (b_p). \]
Thus, writing $\wt{\det}=\det_{\ur} \cdot \det_p$ (whose image is contained in $\ul{\bbZ_p^\times}$), we find that the group 
\[ \l( B_p \times G^{(p)} \r)^1 := \ker \wt{\det} \]
is the stabilizer of $\Ig_{\CS}^1$.  

\subsubsection{$p$-integal moduli action of the unipotent subgroup}\label{sss:explicit-description-unipotent}
It will be useful for computations with the crystalline connection for us to have a more explicit description of the action of $U_p$ on the $p$-integral moduli problem represented by $\Ig_{\CS}$.  We give such a description now; the key point is that any \emph{unipotent} automorphism of the universal cover lifts a unipotent automorphism of the $p$-divisible group $\mu_{p^\infty} \times \bbQ_p/\bbZ_p$ modulo a nilpotent ideal.  

Let $u \in U_p(R) = \wt{\mu_{p^\infty}}(R)$,  and $x \in \Ig_{\CS}(R)$. Write $u=(\zeta_k) \in \wt{\mu_{p^\infty}}(R)$ and let $I$ be any nilpotent ideal of $R$ containing $\zeta_0 - 1$. Then, 
\[ u \mod I = (1, \zeta_1 \mod I, \zeta_2 \mod I, \ldots ) \]
is an element of $U_p^\circ(R/I)$. Now, if $(E, \varphi_p, \varphi_{\bbA_f^{(p)}} )$ is a triple corresponding to $x$ under the $p$-integral moduli interpretation then
\[ u \cdot x = \l(E', \varphi'_p, \varphi'_{\bbA_f^{(p)}} \r) \]
where $E'$ is the Serre-Tate lift from $R/I$ to $R$ of $E_{R/I}$ determined by the isomorphism 
\[ (u \mod I) \circ (\varphi_p)_{R/I} : E_{R/I}[p^\infty] \xrightarrow{\sim} (\mu_{p^\infty} \times \bbQ_p/\bbZ_p)_{R/I},  \]
$\varphi'_p$ is the natural isomorphism 
\[ \varphi'_p: E'[p^\infty]  \xrightarrow{\sim} \mu_{p^\infty} \times \bbQ_p/\bbZ_p \]
and $\varphi'_{\bbA_f^{(p)}}$ is the unique lift of $\varphi_{\bbA_f^{(p)}}|_{R/I}$ from $E_{R/I}=E'_{R/I}$ to $E'$. 

\subsubsection{Moduli action on $\Ig_{\Katz}$}\label{sss:moduli-ig-katz}
The first moduli interpretation of $\Ig_{\Katz}$ leads to an action of $\ul{\bbZ_p^\times} \times G^{(p)}$, where here $\bbZ_p^\times = \Aut(\mu_{p^\infty})$. The second moduli interpretation leads to an action of
\[ \l( M_p^\circ \times G^{(p)} \r)^1 := \ker \wt{\det}|_{M_p^\circ \times G^{(p)}}. \]
These two actions are identified through the isomorphism 
\[ \ul{\bbZ_p^\times} \times G^{(p)} \xrightarrow{\sim} \l( M_p^\circ \times G^{(p)} \r)^1 \]
given by the assignment
\[ a \times (g_\ell)_{\ell \neq p}  \mapsto \begin{pmatrix} a & 0 \\ 0 & a^{-1} \prod_{\ell \neq p} |\det(g_\ell)|_\ell^{-1} \end{pmatrix} \times (g_\ell)_{\ell \neq p}. \] 
 
\subsubsection{Moduli problems with finite prime-to-$p$ level} 
If $K^p \subset \GL_2(\bbA_f^{(p)})$ is a compact open, then we may choose a lattice $\calL$ stabilized by $K^p$, and $N$ such that 
\[ K^p \subset K_N := \ker \GL(\calL) \rightarrow \GL(\calL/N\calL). \]

We may then consider the moduli problems where the prime-to-$p$ a $K^p$ orbit of trivializations $\varphi_{\bbA_f^{(p)}}$, or, what is equivalent (and easier to define precisely), a $K^p/K_N$-orbit of trivializations 
\[ \varphi_{\calL/N}: \underline{\calL/N} \xrightarrow{\sim} E[N]. \]

These finite prime-to-$p$ level moduli problems are also representable, and the infinite prime-to-$p$ level moduli problems are the inverse limit over these. Because the corresponding covers are finite \'{e}tale, we find that the Katz (resp. Caraiani-Scholze, resp. polarized Caraiani-Scholze) moduli problem with finite prime-to-$p$ level $K^p$ is represented by the ring $\bbV_{\Katz}^{K^p}$ (resp. $\bbV_{\CS}^{K^p}$, resp. $\l(\bbV_{\CS}^1\r)^{K^p}$). 

\begin{example} If $\Gamma_1(N)$ denotes the subgroup of $\GL_2(\wh{\bbZ}^{(p)})$ congruent to 
\[ \begin{pmatrix} 1 & * \\ 0 & * \end{pmatrix} \mod N  \]
then we find $\Spf \bbV_{\Katz}^{\Gamma_1(N)}$ represents the moduli problem classifying over $R$ the triples $(E, \wh{\varphi}, P)$ where $E$ is an elliptic curve over $R$, 
$\wh{\varphi}: \wh{E} \xrightarrow{\sim} \wh{\bbG_m}$, and $P \in E[N](R)$ is a point of (fiberwise) exact order $N$. This is the moduli problem most commonly considered in the literature on $p$-adic modular forms. 
\end{example}

\subsection{First presentation of $\Ig_{\Katz}$ as a quotient.}\label{ss.first-pres}
Recall from above that $\Ig_{\CS}$ has a $p$-integral moduli interpretation as parameterizing triples 
\[ (E, \varphi_p, \varphi_{\bbA_f^{(p)}} ) \]
where $E$ is an elliptic curve up to quasi-prime-to-$p$-isogeny, $\varphi_p$ is integral ordinary infinite level structure on $E[p^\infty]$, and $\varphi_{\bbA_f^{(p)}}$ is rational infinite prime-to-$p$ level structure. Then there is a natural projection map to $\Ig_{\Katz}$ given by
\[ (E, \varphi_p, \varphi_{\bbA_f^{(p)}} ) \rightarrow (E, \wh{\varphi}_p, \varphi_{\bbA_f^{(p)}}). \]
where $\wh{\varphi}_p$ is the isomorphism induced by $\varphi_p$ after restriction: 
\[ \wh{E[p^\infty]} \xrightarrow{\varphi_p|_{\wh{E[p^\infty]}}} \wh{\left(\mu_p^{\infty} \times \bbQ_p/\bbZ_p\right)} = \mu_{p^\infty}. \]

It can be verified that this map is an fpqc torsor for the action of 
\[ \begin{pmatrix} 1 & T_p \mu_{p^\infty}  \\ 0 & \ul{\bbZ_p^\times} \end{pmatrix} \subset B^\circ_p \]
on $\Ig_{\CS}$ (cf Lemma \ref{lemma:torsor} below). Thus, $\Ig_{\Katz}$ is an fpqc quotient of $\Ig_{\CS}$ for this group action, and we obtain a residual action of the quotient of the normalizer of this group in $B_p \times G^p$ by the group itself. This induces the action of $\bbZ_p^\times \times G^{(p)}$ given by the first moduli interpretation of $\Ig_{\Katz}$, but nothing further.  

The surprising observation that allows us to construct our $\wh{\bbG_m}$-action on $\Ig_\Katz$ is that, if we restrict this projection map to $\Ig_{\CS}^1$, then the normalizer grows, and we obtain a non-trivial residual action from the unipotent part $\wt{U}_p$. We describe this in the next section. 

\section{The $\wh{\bbG_m}$-action}\label{sec:the-action}

\subsection{Second presentation of $\Ig_{\Katz}$ as a quotient}

If we restrict the projection to $\Ig_{\CS}^1$, then in terms of the second moduli interpretation of $\Ig_{\Katz}$, it is given by
\[ (E, \varphi_p, \varphi_{\bbA_f^{(p)}}) \mapsto (E, \wh{\varphi}_p, \varphi_p^\et, \varphi_{\bbA_f^{(p)}}) \]
where $\wh{\varphi}_p$ is as described above in \ref{ss.first-pres} and $\varphi_p^\et$ is induced by
\[ E[p^\infty]/E[p^\infty]^\circ \xrightarrow{ \varphi_p } \l( \mu_{p^\infty} \times \bbQ_p/\bbZ_p \r) / \mu_{p^\infty} = \bbQ_p/\bbZ_p. \]
The contents of this statement is simply that the pairings used to define $\Ig_{\CS}^1$ and the polarized lifting of $\ref{sss:polarized-lift-Katz}$ are compatible. We then have
\begin{lemma}\label{lemma:torsor}
The projection map $\pi: \Ig_{\CS}^1 \rightarrow \Ig_{\Katz}$ is an fpqc torsor for the action of $U^\circ_p$ on $\Ig_{\CS}^1$.  
\end{lemma}
\begin{proof}
We first show the map is surjective: given 
\[ (E, \wh{\varphi}_p, \varphi_p^\et, \varphi_{\bbA_f^{(p)}}) \in \Ig_{\Katz}(R) \]
consider the induced extension 
\[ 1 \rightarrow \mu_{p^\infty} \rightarrow E[p^\infty] \rightarrow \bbQ_p/\bbZ_p \rightarrow 1. \]
We must find an fpqc extension of $R$ where this extension is split. This is given by the fiber of 
\[ T_p E \rightarrow T_p \l(\bbQ_p/\bbZ_p\r) = \bbZ_p \]
over $1 \in \bbZ_p$, where there is a canonical splitting (recall $T_p E = \Hom(\bbQ_p/\bbZ_p, E[p^\infty])$). Indeed, this fiber is the inverse limit of the fibers of $E[p^n]$ over $1/p^n \in \bbQ_p/\bbZ_p$, each of which is finite flat over $R$ (as a clopen subscheme of $E[p^n]$ that surjects onto $R$). The inverse limit is pro-finite-flat, and thus, in particular, an fpqc cover. 

To see that the map is a torsor, we first observe that from the definition of the action it is clear that $U^\circ_p$ preserves the projection. Moreover, any two pre-images correspond to two different splittings, and thus differ by an element of $U^\circ_p$. Finally, the action of $U^\circ_p$ is faithful because the data away from $p$ already rigidifies the moduli problem. 

\end{proof}

\subsection{Residual action on the quotient}
We write $N(U_p^\circ)$ for the normalizer of $U_p^\circ$ in $\l(B_p \times G^{(p)}\r)^1$. As a consequence of Lemma \ref{lemma:torsor}, which presents $\Ig_{\Katz}$ as a quotient of $\Ig_\CS^1$ by $U_p^\circ$, we find that $N(U_p^\circ)/U_p^\circ$ acts on $\Ig_\Katz$.

Writing out the matrix presentation, we see that $N(U_p^\circ)$ is the subgroup of elements 
\[ g=\begin{pmatrix} a & b \\ 0 & c \end{pmatrix} \times g^{(p)} \in B_p \times G^{(p)}\]
such that $|a|_p=|c|_p$ and $\wt{\det}(g)=1$. This subgroup is isomorphic to
\[  \wt{\mu_{p^\infty}}  \rtimes \l( \ul{\bbZ_p^\times} \times \ul{p^\bbZ} \times G^{(p)} \r) \]
via the map 
\[ u \rtimes \l( (a, p^k) \times (g_\ell)_{\ell \neq p} \r) \mapsto \begin{pmatrix} a p^k & u \\ 0 & a^{-1}p^k \prod_{\ell \neq p} |\det(g_\ell)|_\ell^{-1} \end{pmatrix} \times (g_\ell)_{\ell \neq p}. \]
Thus, 
\[ N(U_p^\circ)/U_p^\circ = \wh{\bbG_m} \rtimes \l( \ul{\bbZ_p^\times} \times \ul{p^\bbZ} \times G^{(p)} \r)   \]
where here we have used that $\wt{\mu_{p^\infty}}/T_p \mu_{p^\infty} = \mu_{p^\infty} = \wh{\bbG_m}.$ 

In particular, we obtain an action of the subgroup $\wh{\bbG_m}$, that is the titular ``unipotent circle action" on $\Ig_\Katz$. 
Comparing with \ref{sss:moduli-ig-katz}, we see that the action of the subgroup $\ul{\bbZ_p^\times} \times G^{(p)}$ agrees with the standard moduli action on $\Ig_{\Katz}$, and thus we immediately obtain the compatibility between the $\wh{\bbG_m}$-action and the standard moduli action described in Theorem \ref{maintheorem:big-hecke}. 

\begin{remark}\newcommand{\diag}{\mathrm{diag}}\label{remark:frobenius}
The subgroup $p^\bbZ$ acts as powers of the classical diamond operator at $p$. Similarly, we obtain the Hecke operator $U_p$ and the canonical Frobenius lift from the Hecke action of $\diag(p,1)$ and $\diag(1,p)$. 
\end{remark}

\subsection{Differentiating the action}\label{subsection:differentiation}
We have now constructed the $\wh{\bbG_m}$-action and proved the desired compatibility with the standard moduli action on $\Ig_{\Katz}.$ To complete the proof of Theorem \ref{maintheorem:big-hecke}, we must prove a claim about the derivative of the $\wh{\bbG_m}$-action. We recall the setup now:

We consider the action map
\[ \wh{\bbG_m} \times \Ig_{\Katz} \rightarrow \Ig_{\Katz}, \]
To differentiate it, we compose with the tangent vector $t\partial_t$ at the identity in $\wh{\bbG_m}$. The latter is given by a map $D \rightarrow \bbG_m$ which in coordinates is
\[ \bbZ_p[t^{\pm 1}] \rightarrow \bbZ_p[\epsilon]/\epsilon^2,  \; t \mapsto 1 + \epsilon . \] 
Thus, the composition of the action map with $t \partial_t$ gives a vector field on $\Ig_\Katz$ described as a map
\[ t_{\bighecke}: D \times \Ig_\Katz \rightarrow \Ig_\Katz. \] 

On the other hand, we have the universal extension 
\[ \calE_{E_\univ[p^\infty], \wh{\varphi}_\univ, \varphi^\et_\univ}: 1 \rightarrow \mu_{p^\infty} \rightarrow E_\univ[p^\infty] \rightarrow \bbQ_p/ \bbZ_p \rightarrow 1 \]
(cf. (\ref{equation:igusa-data-extension})), and, as explained in \ref{subsection:the-canonical-trivialization}, this extension gives rise to a differential ${d\tau \in \Omega_{\Ig_\Katz}.}$ To complete our proof of Theorem \ref{maintheorem:big-hecke}, we show
\begin{theorem}\label{theorem:differentiation-action-on-igusa}
Notation as above,
\[ d\tau(t_\bighecke) = 1 \]
\end{theorem}

\begin{proof}
It suffices to work over $\bbZ/p^n$ for arbitrary $n$. We abbreviate $S=\Ig_{\Katz}|_{\bbZ/p^n}$ and $R=\bbV_\Katz/p^n$ so that $S=\Spec R$. We write $\pi: E \rightarrow S$ for the universal elliptic curve up to prime-to-$p$-isogeny over $S$ and $\wh{\varphi}$, $\varphi^\et$ for the universal trivializations of $\wh{E}$ and $E[p^\infty]/\wh{E}$. 

We recall the definition of $d\tau$ (mod $p^n$): we have the canonical extension 
\[ \calE_{E[p^\infty], \wh{\varphi}, \varphi^\et}: 1 \rightarrow \mu_{p^\infty} \rightarrow E[p^\infty] \rightarrow \bbQ_p/\bbZ_p \rightarrow 1 \]
and the induced trivialization $\omega_\can$, $u_\can$ of the vector bundle $\Lie \bE(E[p^\infty])$ with connection $\nabla_\crys$. Then $d\tau$ is defined by the equation
\[ \nabla_\crys (\omega_\can) = u_\can d\tau.\]

As in \ref{sss:connections-and-vector-fields}, we write $\nabla_{\crys, t_\bighecke}$ for the isomorphism 
\[ t_\bighecke^* \Lie \bE(E[p^\infty]) \xrightarrow{\sim} 0^* \Lie \bE(E[p^\infty]) \]
over $D\times S$ induced by $\nabla_\crys$. In light of (\ref{equation:connection-isomorphism}), it suffices to show that 
\begin{equation}\label{equation:suffices-to-show-differentiating} \nabla_{\crys, t_\bighecke}(\omega_\can)= \omega_\can + \epsilon \cdot u_\can. \end{equation}

Now, we have that ${t_\bighecke}=(1+\epsilon) \cdot 0 $ via the $\wh{\bbG_m}$ action, where we view the tangent vectors ${t_\bighecke}$ and $0$ as $R[\epsilon]$-points of $S$ and $(1+\epsilon)$ as an $R[\epsilon]$-point of $\widehat{\bbG_m}$. 
Writing
\[ (E_{0}, \wh{\varphi}_0, \varphi^\et_0, \alpha_0) \]
for the quadruple classified by $0$ and 
\[ (E_{\bighecke}, \wh{\varphi}_\bighecke, \varphi^\et_\bighecke, \alpha_{\bighecke}) \]
for the quadruple classified by $t_\bighecke$, we have
\begin{equation}\label{equation:multiplication-big-hecke} (1+\epsilon) \cdot (E_{0}, \wh{\varphi}_0, \varphi^\et_0, \alpha_0) = (E_{\bighecke}, \wh{\varphi}_\bighecke, \varphi^\et_\bighecke, \alpha_{\bighecke}) \end{equation}

In particular, $\nabla_{\crys, t_\bighecke}$ is identified with the Messing isomorphism
\[ \Lie \bE( E_\bighecke[p^\infty]^\vee )\xrightarrow{\sim}  \Lie \bE(E_0[p^\infty]^\vee) \]
induced by the isomorphism
\[ E_\bighecke \mod \epsilon = E_0 \mod \epsilon \]
given by $1+\epsilon=1 \mod \epsilon$ and (\ref{equation:multiplication-big-hecke}). 

To compute this, we pass to the flat cover 
\[ S_\infty = \Spec \l(\VV_{\CS}^1/p^n\r)[\epsilon, (1+\epsilon)^{1/p^\infty}]/\epsilon^2. \]
Over $\l(\bbV_{CS}^1/p^n\r)[\epsilon]$ and thus over $S_\infty$, we have a canonical splitting of $\calE_{E_0[p^\infty], \wh{\varphi}_0, \varphi^\et}$ which gives an isomorphism
\[ \varphi_0: E_0[p^\infty]_{R_\infty} \xrightarrow{\sim} \widehat{\GG_m} \times \QQ_p/\ZZ_p. \]

If we let 
\[  g_\epsilon := \begin{pmatrix} 1 & (1+\epsilon, (1+\epsilon)^{1/p}, \ldots) \\ 0 & 1 \end{pmatrix} \in U_p(S_\infty), \]
our description of the unipotent action in \ref{sss:explicit-description-unipotent} then shows that over $S_\infty$, $E_\bighecke$ is the Serre-Tate lift to $S_\infty$ corresponding to the isomorphism 
\[  g_\epsilon \circ \varphi_0: E_0|_{S_\infty/\epsilon}[p^\infty] \rightarrow \wh{\bbG_m} \times \bbQ_p/\bbZ_p. \]
Thus, the Messing isomorphism in the canonical basis is identified over $S_\infty$ with the map
\[ \Lie \bE (\bbQ_p/\bbZ_p \times \wh{\bbG_m}) \rightarrow \Lie \bE(  \bbQ_p/\bbZ_p \times \wh{\bbG_m} ) \]
induced by $g_{\epsilon}^{t}$. If we write this in the canonical basis we get a  map
\[  \bbG_a \frac{dt}{t} \times \bbG_a t \partial_t \rightarrow \bbG_a \frac{dt}{t} \times \bbG_a t \partial_t,\]
and, by Lemma \ref{lemma:vector-extension-map}, it is given by 
\[ \begin{pmatrix} 1 & 0 \\ \epsilon & 1\end{pmatrix}. \]
By construction, these bases are identified with the bases $\omega_\can, u_\can$, and thus we obtain equation (\ref{equation:suffices-to-show-differentiating}), concluding the proof. 

\end{proof}

\section{Local expansions}\label{sec:local-expansions}
In this section we compute the $\wh{\bbG}_m$-action on Serre-Tate ordinary and cuspidal $q$-expansions (i.e. on the formal neighborhood of an $\overline{\bbF_p}$-point of $\Ig_{\Katz}$ and on the punctured formal neighborhood of a cusp). In both cases, the action is given by a simple multiplication of the canonical coordinate (cf. Corollaries \ref{corollary:action-on-Tate} and \ref{corollary:action-on-Serre-Tate} below for precise statements). 

In fact, both computations are special cases of a more general statement, Theorem \ref{theorem:computation-of-action} below, which computes the action on a point whenever the associated $p$-divisible group is a Kummer $p$-divisible group (as defined in \ref{subsection:kummer-p-divisible-groups}). 

The statement of Theorem \ref{theorem:computation-of-action} begs the question: is the $p$-divisible group of the universal elliptic curve over $\Ig_\Katz$ (with its extension structure) a Kummer $p$-divisible group? Indeed, one can find claims that the local Serre-Tate coordinates extend to a function $q$ on $\Ig_{\Katz}$, which would imply this group was Kummer. However, these claims are flawed, and indeed the universal extension is not Kummer; we take a brief detour in \ref{subsec:not-kummer} to dispel these myths.

\subsection{Another moduli interpretation for $\Ig_\Katz$}\label{ss:another-moduli-katz}
To state our general computation cleanly, we introduce a third moduli interpretation for $\Ig_{\Katz}$ that puts the emphasis on $p$-divisible group and its extension structure (as in Remark \ref{remark:igusa-data-extension}). 

Let $R$ be a $p$-adically complete ring and let $\pi$ be topologically nilpotent for the $p$-adic topology on $R$. Combining the second moduli interpretation of $\Ig_{\Katz}$ (cf. \ref{sss:moduli-ig-katz}) and Serre-Tate lifting theory (cf. \ref{ss:serre-tate}), we obtain an identification
\[ \Ig_{\Katz}(R) = \{ (E_0, \calE, \psi, \varphi_{\bbA_f^{(p)}} )\} / \sim \]
where $E_0$ is an elliptic curve up to quasi-prime-to-$p$-isogeny over $R/\pi$, $\calE$ is an extension of $p$-divisible groups over $R$
\[ \calE: 1 \rightarrow \mu_{p^\infty} \rightarrow G_{\calE} \rightarrow \bbQ_p/\bbZ_p \rightarrow 1, \]
$\psi: E_0[p^\infty] \xrightarrow{\sim} G_{\calE}|_{R/\pi}$ is an isomorphism, and $\varphi_{\bbA_f^{(p)}}$ is infinite prime-to-$p$ level structure on $E_0$, all subject to a compatibility with the Weil pairing as in the second moduli interpretation of $\Ig_{\Katz}$.

\subsection{Computing the action on Kummer extensions}\label{ss:computing-the-action}
Recall from \ref{subsection:kummer-p-divisible-groups} that there is a Kummer construction which, given $q \in \GG_m(R)$ produces an extension of $p$-divisible groups over $R$
\[ \calE_q: 1 \rightarrow \mu_{p^\infty} \rightarrow G_q \rightarrow \QQ_p/\ZZ_p \rightarrow 1. \] 

Using the explicit description of the $\wh{\bbG_m}$-action in \ref{sss:explicit-description-unipotent}, we find
\begin{theorem}\label{theorem:computation-of-action}
Suppose $\zeta \in \widehat{\GG_m}(R)$ and $\pi\in R$ is such that $\zeta \equiv 1 \mod \pi$, and $x \in \Ig_{\Katz}(R)$ is represented by the quadruple $(E_0, \calE_q, \psi, \varphi_{\bbA_f^{(p)}})$ (in the sense of \ref{ss:another-moduli-katz})  
for some $q \in \GG_m(R)$.  Then 
\[ \zeta \cdot x = (E_0, \calE_{\zeta^{-1} q}, \psi', \varphi_{\bbA_f^{(p)}}) \]
where $\psi'$ is the composition of $\psi$ with the canonical identification 
\[ \calE_q|_{R/\pi}=\calE_{\zeta^{-1}q}|_{R/\pi} \]
coming from $q \equiv \zeta^{-1}q \mod \pi$. 
\end{theorem}
\begin{proof}
If we write  $y$ for the point represented by $(E_0, \calE_{\zeta^{-1} q}, \psi', \varphi_{\bbA_f^{(p)}})$ then it suffices to show that over the extension $R_\infty=R[q^{1/p^\infty}, \zeta^{1/p^\infty}]$, there are lifts $\tilde{x}$ and $\tilde{y}$ of $x$ and $y$ to $\Ig_{\CS}^1(R_\infty)$ and a lift $\tilde{\zeta}$ of $\zeta$ in $\wt{\mu_{p^\infty}}(R_\infty)$ such that $\tilde{\zeta}\cdot \tilde{x}=\tilde{y}$. The desired lifts are given by the splittings $1/p^n \mapsto (q^{1/p^n}, 1/p^n)$ and $1/p^n \mapsto \zeta^{-1/p^n} q^{1/p^n}$ of $\calE_q$ and $\calE_{\zeta^{-1}q}$, respectively, and $\tilde{\zeta}=(\zeta^{1/p^n})_n$. That $\tilde{\zeta} \cdot \tilde{x}_1 = \tilde{x}_2$ then follows from commutativity of the following diagram mod $\zeta - 1$:
\[ \xymatrix{
G_q \ar[r]^{=} & G_{\zeta^{-1}q} \\
 & \\
\mu_{p^\infty} \times \bbQ_p/\bbZ_p \ar[uu]^{1/p^n \mapsto (q^{1/p^n}, 1/p^n)} \ar[r]_{\left(\begin{smallmatrix}1 & \tilde{\zeta} \\ 0 & 1 \end{smallmatrix}\right)} & \mu_{p^\infty} \times \bbQ_p/\bbZ_p \ar[uu]_{1/p^n \mapsto (\zeta^{-1/p^n}q^{1/p^n}, 1/p^n)}
}
\]
\end{proof}

\subsection{Action on Serre-Tate coordinates}
We now compute the local expansion in the formal neighborhood of an $\overline{\bbF_p}$-point of $\Ig_\Katz$. So, fix
\[ (E_0, \wh{\varphi}_0, \varphi^\et_0, \varphi_{\bbA_f^{(p)},0}) \in \Ig_{\Katz}(\overline{\bbF_p}). \]

It follows from the work of Serre-Tate (as described in \cite{katz:serre-tate}) that the formal neighborhood of this point is $\Spf R$ for $R$ a smooth complete 2-dimensional local ring over $W(\overline{\bbF_p})$. The data attached to the induced $\Spf R \rightarrow \Ig_{\Katz}$
\[ (E, \wh{\varphi}, \varphi^\et, \varphi_{\bbA_f^{(p)}}) \]
identifies $E$ with the universal deformation of $E_0$, and the level structures with their unique deformations. 

Moreover, the Serre-Tate coordinate $q \in 1 + \frakm_R = \wh{\bbG_m}(R)$ of the induced extension
\[ \calE_{E[p^\infty], \wh{\varphi}, \varphi^\et, \varphi_{\bbA_f^{(p)}}}: 1 \rightarrow \mu_{p^\infty} \rightarrow E[p^\infty] \rightarrow \bbQ_p/\bbZ_p \rightarrow 1 \]
gives an isomorphism
\[ q: \Spf R \xrightarrow{\sim} \wh{\bbG_m}. \]
Recall from Remark \ref{remark:alternate-construction} that the extension $\calE_{E[p^\infty], \wh{\varphi}, \varphi^\et, \varphi_{\bbA_f^{(p)}}}$ is equal to the Kummer extension $\calE_{q^{-1}}$. Thus, invoking Theorem \ref{theorem:computation-of-action}, we deduce 
\begin{corollary}\label{corollary:action-on-Serre-Tate} The $\wh{\bbG_m}$-action on $\Ig_\Katz$ preserves $\Spf R \subset \Ig_{\Katz}$, and the Serre-Tate coordinate
\[ q: \Spf R \xrightarrow{\sim} \wh{\bbG_m} \]
identifies the $\wh{\bbG_m}$-action on $\Spf R$ with multiplication on $\wh{\bbG_m}$. 
\end{corollary}

\begin{remark} The choice of a basis $x \in T_p E_0(\overline{\bbF_p})$ is equivalent the to the choice $\varphi^\et_0$, and via the Weil pairing, gives a corresponding choice of $\wh{\varphi}_0$, 
\[ ( \bullet , x): \wh{E_0} \rightarrow \mu_{p^\infty}. \]
This is how the data is presented, e.g., in \cite{katz:serre-tate}. 
\end{remark}

\subsection{The universal extension over $\Ig_\Katz$ is not Kummer}\label{subsec:not-kummer}
There are some instances in the literature where it is claimed that the local Serre-Tate coordinates $q$ extend to a \emph{function}\footnote{As explained by Chai \cite{chai:canonical}, there is a natural map of \'{e}tale sheaves that interpolates the local Serre-Tate coordinates.} on all of $\Ig_{\Katz}$ (or a rigid analytic incarnation). These claims are based on a misapplication of a theorem of Serre-Tate (cf. Katz \cite{katz:serre-tate}) classifying extensions of $\bbQ_p/\bbZ_p$ by $\wh{\bbG_m}$ over Artin local rings to the more general setting of rings where $p$ is nilpotent, where the classification no longer holds. 

In fact, the claim itself is incorrect, as we now explain. Already in characteristic $p$ the existence of such a global coordinate boils down to the claim that for $E_\univ$ the universal curve over $\Ig_{\Katz, \overline{\bbF_p}}$, the canonical extension 
\[ \calE_{E_\univ, \wh{\varphi}_\univ, \varphi_{\bbA_f^{(p)}}}: 1 \rightarrow \mu_{p^\infty} \rightarrow E[p^\infty] \rightarrow \bbQ_p/\bbZ_p \rightarrow 1 \] 
is a Kummer $p$-divisible group (as defined in \ref{subsection:kummer-p-divisible-groups}).

This is not the case: if it were, then $\calE_{E_\univ, \wh{\varphi}_\univ, \varphi_{\bbA_f^{(p)}}}$ would be of the form $\calE_{q_\univ}$ for some $q_\univ \in \bbV_{\Katz, \overline{\bbF_p}}^\times$, and thus would only depend on $\wh{\varphi}|_{\wh{E}[p^n]}$ for some sufficiently large $n$ -- indeed, $\bbV_{\Katz, \overline{\bbF_p}}$ is the colimit of the rings representing the moduli problem where the level structure at $p$ is an isomorphism of $\wh{E}[p^n]$ with $\mu_{p^n}$. In particular, $q_\univ$ and thus the isomorphism class of $\calE_{E_\univ, \wh{\varphi}_\univ, \varphi_{\bbA_f^{(p)}}}$ would be preserved under the action of $1+p^n \in \bbZ_p^\times$. 

This gives a contradiction: Fix a point in $x \in \Ig_{\Katz}(\overline{\bbF_p})$ as in the previous section, and consider its orbit $O=\bbZ_p^\times \cdot x$ (which is isomorphic to the profinite set $\bbZ_p^\times$ as a scheme over $\overline{\bbF_p}$ since the action is free on $\Ig_{\Katz}(\overline{\bbF_p})$). Without using $q_\univ$, we find that the Serre-Tate coordinate \emph{does} exist on a formal neighborhood $O^\vee$ of $O$ in $\Ig_{\Katz, \overline{\bbF_p}}$ as a map 
\begin{equation}\label{eqn:extended-serre-tate}
 O^\vee \rightarrow \wh{\bbG_m}, \end{equation}
sending a point in an Artin local $\overline{\bbF_p}$-algebra to the Serre-Tate coordinate of the extension $\calE_{E_\univ, \wh{\varphi}_\univ, \varphi_{\bbA_f^{(p)}}}$ at that point. A short computation from the definition of this extension shows the map (\ref{eqn:extended-serre-tate}) is $\bbZ_p^\times$-equivariant, where the action of $z \in \bbZ_p^\times$ on the right is by the automorphism $z^2$ of $\wh{\bbG_m}$. In particular, $1+p^n$ acts non-trivially on the right-hand side.  On the other hand, since our assumption that the extension was Kummer implies (as argued above) that the element $1+p^n$ preserves the isomorphism class of $\calE_{E_\univ, \wh{\varphi}_\univ, \varphi_{\bbA_f^{(p)}}}$, the action of $1+p^n$ must also preserve the map (\ref{eqn:extended-serre-tate}), a contradiction. 

\begin{remark}
One can argue similarly using the Tate curve instead of a Serre-Tate ordinary expansion. 
\end{remark}

\subsection{Action on the Tate curve and $q$-expansions} \label{ss:cusps}
Let 
\[ R = \left( \colim_{(N,p)=1}  W(\overline{\bbF_p})((q^{1/N})) \right)^\wedge \]
and consider the Tate curve $\Tate(q)$  over $R$. We have the canonical trivialization 
\[ \varphi_{\can}: \wh{\Tate(q)} \xrightarrow{\sim} \wh{\bbG_m}, \]
and, if we fix a compatible system of prime-to-$p$ roots of unity $\zeta_N$, we obtain a basis $(\zeta_N)_N, (q^{1/N})_N$ for $T_{\wh{\bbZ}^{(p)}} \Tate(q)$ over $R$ and thus a trivialization $\varphi_{\bbA_f^{(p)}}$ of the prime-to-$p$ adelic Tate module.

The \emph{cusps} of $\Ig_\Katz$ are the $R$-points in the $\GL_2(\wh{\bbZ}^{(p)})$-orbit of 
\[ \l( \Tate(q), \wh{\varphi}_\can, \varphi_{\bbA_f^{(p)}} \r). \]
For $g \in \bbV_{\Katz,A}$ and a cusp $c$, we call the element $c(g) \in A \tensorhat R$ the $q$-expansion of $g$ at $c$. We find

\begin{corollary}\label{corollary:action-on-Tate}
If $c$ is a cusp of $\Ig_{\Katz}$ and $g \in \bbV_{\Katz, A}$ has $q$-expansion at $c$ 
\[ \sum_{k \in \bbZ_{(p)}} a_{k} q^k ,\]
then, for $\zeta \in \wh{\bbG_m}(A)$, 
\[ \zeta \cdot g := (\zeta^{-1} \cdot)^* g \]
has $q$-expansion at $c$
\begin{equation}\label{equation:q-expansion-action} \sum_{k \in \bbZ_{(p)}} a_{k} (\zeta q)^k = \sum_{k \in \bbZ_{(p)}} \zeta^k a_{k}  q^k\end{equation}
(where the powers $\zeta^k$ make sense because $k \in \bbZ_{(p)} \subset \bbZ_p$). 
\end{corollary}
\begin{proof}
The $q$-expansion in \ref{equation:q-expansion-action} is the image of $c(g)$ under the the map 
\begin{align*} \gamma_{\zeta}: R \tensorhat{A} & \rightarrow R\tensorhat{A} \\
 q^k & \mapsto (\zeta q)^k. \end{align*} 
Thus, because the action of $\wh{\bbG_m}$ commutes with the action of $\GL_2(\wh{\bbZ}^{(p)})$ (because the latter is in the kernel of $\det_\ur$), we may assume our cusp is given by the triple 
\[ \l( \Tate(q), \wh{\varphi}_\can, \varphi_{(\zeta_N)_N, (q^{1/N})_N} \r). \]
Then, it follows from Theorem \ref{theorem:computation-of-action} that 
\[ \zeta^{-1} \cdot \l( \Tate(q), \wh{\varphi}_\can, \varphi_{(\zeta_N, q^{1/N})} \r) = \l( \Tate(\zeta q), \wh{\varphi}_\can, \varphi_{(\zeta_N, \zeta^{1/N}q^{1/n})} \r). \]
This is the base change of 
\[ \l( \Tate(q), \wh{\varphi}_\can, \varphi_{(\zeta_N)_N, (q^{1/N})_N} \r) \]
through $\gamma_\zeta$, and thus we conclude. 
\end{proof}

\begin{remark}
Using Corollary \ref{corollary:action-on-Tate} over $A=\bbZ_p[\epsilon]$, we find that if we differentiate the $\wh{\bbG_m}$-action along $t \partial_t$ in the sense of \ref{subsection:differentiation}, the induced operator on $q$-expansions is $-q\partial_q = - \theta$ (we get a minus sign because to get the derivation in \ref{subsection:differentiation} we did not compose with an inverse as we have to obtain the natural left action on functions). 
\end{remark}

\section{Fourier transform and the algebra action}\label{sec:algebra-action}

We now explain how the $\wh{\bbG_m}$-action induces an action of $\Cont(\bbZ_p, \bbZ_p)$ on $\bbV_{\Katz}$ via $p$-adic Fourier theory. We then explain the relation between our approach and a classical construction of the algebra action (which is in fact equivalent to the $\wh{\bbG_m}$-action) due to Gouvea \cite{gouvea:arithmetic-of-p-adic-modular-forms} in terms of the exotic isomorphisms of Katz \cite{katz:p-adic-interpolation-of-real-analytic-eisenstein-series}. 

\subsection{The action of $\Cont(\bbZ_p, R)$}\label{ss:algebra-action}
For any $p$-adically complete ring $R$, the action map for the $\wh{\bbG_m}$-action is described by a continuous map
\[ a^*: \VV_{\Katz, R} \rightarrow R[[T]] \tensorhat_{R} \VV_{\Katz, R}. \]
The left action of $\wh{\bbG_m}(R)$ on $\bbV_{\Katz,R}$ is by $\zeta \cdot g = (\zeta^{-1} \cdot)^*g$, and we can express this via the action map: if we view $\zeta \in \wh{\bbG_m}(R)$ as the map $R[[T]] \rightarrow R$ given by $T \mapsto \zeta$, and write $\iota$ for the inverse map $\wh{\bbG_m} \rightarrow \wh{\bbG_m}$, then $\zeta \cdot g$ is the image of 
\[ (\iota\times \Id)^* a^* g \]
under the induced map 
\[ \zeta_{\bbV_{\Katz, R}}: R[[T]] \tensorhat_R \bbV_{\Katz, R} \rightarrow \bbV_{\Katz, R}. \]

More generally, if we identify $R[[T]]$ with the continuous $R$-linear dual of $\Cont(\ZZ_p, R)$ via the Amice transform, we obtain an $R$-linear map 
\begin{align*} \Cont(\ZZ_p, R) \times \bbV_{\Katz, R} & \rightarrow \bbV_{\Katz, R} \\
(f, g) &\mapsto f \cdot g := \langle f, (\iota\times \Id)^* a^* g \rangle. 
\end{align*}
If we let $\chi_\zeta$ be the $R$-valued character of $\bbZ_p$ given by $\chi_\zeta(a) = \zeta^a$, viewed as an element of $\Cont(\bbZ_p, R)$, we find
\[ \chi_\zeta \cdot g = \zeta \cdot g. \]
That we have an action of $\widehat{\GG_m}$ to begin with is equivalent to this being an algebra action of $\Cont(\ZZ_p, R)$ on $\bbV_{\Katz, R}.$  

More generally, we find (using the notation for cusps as in \ref{ss:cusps} above):

\begin{theorem}\label{theorem:algebra-action-q-expansions}
If $g \in \bbV_{\Katz, R}$ has $q$-expansion at a cusp $c$
\[ c(g)=\sum_{k \in \bbZ_{(p)}} a_{k} q^k, \]
then the $q$-expansion of $f \cdot g$ is
\[ c(f \cdot g) = \sum_{k \in \bbZ_{(p)}} f(k) a_{k} q^k. \]
\end{theorem}
\begin{proof}
Because for a general $R$, 
\[ \bbV_{\Katz, R} = \bbV_{\Katz, \bbZ_p} \tensorhat R \textrm{ and } \Cont(\bbZ_p, R) = \Cont(\bbZ_p, \bbZ_p) \tensorhat R ,\]
it suffices to verify this for $R=\bbZ_p$. Then, because the base change is injective on the target ring for $q$-expansions, it suffices work over $R=\calO_{\bbC_p}$. Moreover it suffices to verify the identity for the action of locally constant functions on $\bbZ_p$, which are dense in continuous functions. But for any locally constant function, some $\calO_{\bbC_p}$-multiple can be written as a linear combination of characters, and thus Corollary \ref{corollary:action-on-Tate} gives the result for a multiple of each locally constant function. Since the total $q$-expansion map (i.e. the product of the $q$-expansion maps over all cusps) is injective and the target ring is torsion-free over $\calO_{\bbC_p}$, we find that the result holds for all locally constant functions.
\end{proof}

\subsection{The Frobenius and Gouvea's construction}\label{ss:gouvea}

\newcommand{\diag}{\mathrm{diag}}
As noted in Remark \ref{remark:frobenius}, the diagonal quasi-isogeny $\diag(p^n,1) \subset M_p \subset B_p$ acting on the moduli problem $\Ig^1_\CS$ induces a map $\Ig_\Katz \rightarrow \Ig_\Katz$ because it conjugates $U_p^\circ $ to $p^n  U_p^\circ$. In particular this map  factors as 
\[ \Ig_{\Katz} = U_p^\circ\bs \Ig_{\CS} \xrightarrow{\sim} p^n U_p^\circ \bs \Ig_{\CS}  \rightarrow U_p^\circ \bs \Ig_\CS  = \Ig_\Katz \]
where the first map is the isomorphism induced by $\diag(p^n, 1)$ and the second map is the projection. 

In fact, this map is the canonical Frobenius lift (sending $E$ to $E/\mu_{p^n}$ with the induced level structures), and the isomorphism that factors it is one of Katz's exotic isomorphisms (cf. \cite[Lemma 5.6.3]{katz:p-adic-interpolation-of-real-analytic-eisenstein-series}). The action of $\mu_{p^n} \subset \wh{\bbG_m}$ on $\Ig_{\Katz}$ is identified through the exotic isomorphism with the action of $\mu_{p^n} = U_p^\circ / p^n U_p^\circ$ on $p^n U_p^\circ \bs \Ig_{\CS}$. The latter space has a moduli interpretation relative to $\Ig_\Katz$ as parameterizing splittings of the canonical extension 
\[ 1 \rightarrow \mu_{p^n} \rightarrow E[p^n] \rightarrow 1/p^n\bbZ_p / \bbZ_p \rightarrow 1, \]
and the action of $\mu_{p^n}$ is just by changing the splitting. In particular, one can compute the action on $q$-expansions using this relative moduli interpretation to recover the action of all finite order characters, which, combined with the $q$-expansion principle, is enough to produce the full algebra action of Theorem \ref{theorem:algebra-action-q-expansions}. This is essentially what is done by Gouvea \cite[III.6.2]{gouvea:arithmetic-of-p-adic-modular-forms}, who first constructed this action (which he thought of as a twisting measure, much in the spirit of our application in the next section).

\begin{remark}
The moduli problem $\Ig_{\CS}^1$ is the inverse limit of the moduli problems parameterizing splittings at level $p^n$ over $\Ig_\Katz$, and thus is the inverse limit along the canonical Frobenius lift of $\Ig_\Katz$. As observed by Caraiani-Scholze \cite{caraiani-scholze:generic}, this implies that 
$\bbV_{\CS}^1$ is isomorphic to the Witt vectors of the perfection of $\bbV_{\Katz, \bbF_p}.$
\end{remark}

\section{Eisenstein measures}\label{section:eisenstein-measures}
In a series of papers (\cite{katz:l-via-moduli, katz:eisenstein-measure-and-p-adic-interpolation, katz:p-adic-interpolation-of-real-analytic-eisenstein-series}), Katz introduced increasingly general \emph{Eisenstein measures} with values in $\VV_\Katz$ interpolating Eisenstein series. These Eisenstein measures specialize at the cusps and ordinary CM points to $p$-adic L-functions interpolating $L$-values attached to characters of the id\`{e}le class group of $\bbQ$ or an imaginary quadratic extension.  

The papers \cite{katz:l-via-moduli,katz:eisenstein-measure-and-p-adic-interpolation} are concerned with single variable $p$-adic $L$-functions, whereas \cite{katz:p-adic-interpolation-of-real-analytic-eisenstein-series} gives two variable $L$-functions by interpolating not just holomorphic Eisenstein series but also certain real analytic Eisenstein series. 

In this section, we explain how ``half" of the two-variable measure can be produced by a type of convolution of the single-variable measure with the $\widehat{\GG_m}$-action. To keep the exposition clear, we work at level $K^p=\GL_2 (\widehat{\ZZ}^{(p)})$ away from $p$. 

\begin{remark}
The real analytic Eisenstein series are related to the holomorphic Eisenstein series by iterated application of the differential operator $\theta$ -- thus, we can summarize the difference between our approach and that of Katz by saying that instead of applying $\theta$ and then interpolating, we have first interpolated $\theta$ and then applied this interpolated operator to the holomorphic Eisenstein measure. 
\end{remark}

\subsection{Measures}
For $R$ a $p$-adically complete $\ZZ_p$-algebra and $X$ a profinite set, an $R$-valued measure on $X$ is an element 
\[ \mu \in \Hom_{\ZZ_p}(\Cont(X, \ZZ_p), R). \]
Note that such a $\mu$ is automatically continuous for the $p$-adic topology on $\Cont(X, \ZZ_p)$ and $R$. In fact, the stronger basic congruence property holds: if $f \equiv g \mod p^n$, then $\mu(f) \equiv \mu(g) \mod p^n$ -- this observation is at the heart of the application of measures to $p$-adic $L$ functions. 

\begin{remark}\label{remark:distribution-equals-measure}
An $R$-valued distribution is an $R$-valued functional on the space of locally constant functions on $X$, $C^\infty(X,\ZZ_p)$. The space $C^\infty(X,\ZZ_p)$ is dense in $\Cont(X, \ZZ_p)$, thus when $R$ is $p$-adically complete a distribution automatically completes to a measure, and the two notions are equivalent. We use this below. 
\end{remark}

\subsubsection{Measures on products}

\begin{proposition}\label{prop:bilinear-form-to-measure}
Let $X$ and $Y$ be profinite sets, and $R$ a $p$-adically complete $\ZZ_p$-algebra. If $(\,,\,)$ is an $R$-valued $\ZZ_p$-bilinear pairing on $\Cont(X, \ZZ_p) \times \Cont(Y, \ZZ_p)$, then there is a unique $R$-valued measure $\mu$ on $X \times Y$ such that for $f \in \Cont(X, \ZZ_p)$ and $g \in \Cont(Y, \ZZ_p)$, 
\begin{equation} \label{eqn:product-property-measure} \mu(fg)=(f, g). \end{equation}
\end{proposition}
\begin{proof}
By Remark \ref{remark:distribution-equals-measure}, it suffices to construct a functional on $C^\infty(X \times Y, \ZZ_p)$ satisfying (\ref{eqn:product-property-measure}), and then verify that (\ref{eqn:product-property-measure}) holds for any continuous $f$ and $g$ and the unique extension of that distribution to a measure. So, suppose we have constructed a measure $\mu$ such that (\ref{eqn:product-property-measure}) holds for $f$ and $g$ locally constant. Then, for any continuous $f$ and $g$ and $n \in \ZZ_{>0}$, pick $f_n$ and $g_n$ locally constant such that $f \equiv f_n \mod p^n$ and $g \equiv g_n \mod p^n$. Then 
\[ fg \equiv f_n g_n \mod p^n \]
and thus
\[ \mu (fg) \equiv \mu(f_n g_n)  \equiv (f_n,g_n) \equiv (f,g) \mod p^n \]  
and we conclude $\mu(fg) = (f,g)$. 

Thus it remains to construct the distribution and show that it is unique. The bilinear pairing $(\;,\;)$ induces a functional on $C^\infty(X, \ZZ_p) \otimes_{\ZZ_p} C^\infty(Y,\ZZ_p)$, thus to conclude, it suffices to show that the product map  
\[ C^\infty(X, \ZZ_p) \otimes_{\ZZ_p} C^\infty(Y,\ZZ_p) \rightarrow C^\infty(X \times Y, \ZZ_p) \]
is an isomorphism: For any profinite set $W$, $C^\infty(W, \ZZ_p)$ is the colimit over finite coverings $\calU = \{U_1,...,U_n\}$ of $W$ by disjoint compact opens of $C^\infty_U(W, \ZZ_p)$, the space of functions constant on each of the $U_i$. In particular, if $\calU=\{U_1,...,U_n\}$ is such a cover of $X$ and $\calV=\{V_1,...,V_m\}$ is such a cover of $Y$ then 
\[ \calU \times \calV := \{ U_i \times V_j\} \]
is such a cover of $X \times Y$, and the covers of this form are cofinal for covers of $X \times Y$ by disjoint compact opens. Considering the basis of characteristic functions, we find that the product map induces an isomorphism 
\[ C^\infty_{\calU}(X, \ZZ_p) \otimes_{\ZZ_p} C^\infty_{\calV}(Y, \ZZ_p) \rightarrow C^\infty_{\calU \times \calV}(X \times Y, \ZZ_p) \]
and passing to the colimit over covers $\calU$ and $\calV$, we conclude. 
\end{proof}

\begin{example}
If $\nu_X$ and $\nu_Y$ are $R$-valued measures on $X$ and $Y$, respectively, then 
\[ (f, g) \mapsto \nu_X(f) \nu_Y(g) \]
is a bilinear form and the resulting measure $\mu$ on $X \times Y$ is the product measure. 
\end{example}

\subsection{Katz's Eisenstein measures}
In this section, we write $\bbV = \bbV_{\Katz}^{\GL_2(\wh{\bbZ}^{(p)})}$ for the ring representing the Katz moduli problem with no prime-to-$p$ level structure. 
 
\subsubsection{Single variable measures} 
In \cite[XII]{katz:l-via-moduli}, Katz introduced the single variable Eisenstein measures
\begin{align*}
 \mu^{(a)}: \Cont(\ZZ_p, \ZZ_p) \rightarrow & \VV 
\end{align*}
characterized by the moments 
\[ \mu^{(a)}(z^{k-1}) = (1-a^k) 2 G_k \] 
where $2 G_k$ for $k \geq 2$ even is the Eisenstein series with $q$-expansion
\[ \zeta(1-k) +  2\sum_{n=1}^\infty q^n \cdot  \sum_{d|n} d^{k-1} \]
and $2G_k=0$ for $k$ odd.  

\begin{remark}\label{remark:interpolation-locally-constant}
For $k \geq 2$ even and $f$ an even locally constant function on $\ZZ_p$, $\mu^{(a)}$ satisfies the following additional interpolation property \cite[Corollary 3.3.8]{katz:eisenstein-measure-and-p-adic-interpolation}\footnote{In this reference, $\mu^{(a)} = H^{a,1}$ for $N=1$, except for a shift from $k$ to $k+1$.}: 
\[ \mu^{(a)}(z^{k-1} f) = (1-a^k) 2 G_{k,f} \]
where $2G_{k,f}$  has $q$-expansion
\[ L(1-k, f) + 2\sum_{n=1}^\infty q^n \cdot  \sum_{d|n} d^{k-1}f(d). \]
The series for general locally constant $f$ is also described in loc. cit.
\end{remark}

\subsubsection{Two variable measures}
In \cite{katz:p-adic-interpolation-of-real-analytic-eisenstein-series}, Katz introduced the two variable Eisenstein-Ramanujan measures

\begin{align*}
 \mu^{(a,1)}: \Cont(\ZZ_p \times \ZZ_p, \ZZ_p) \rightarrow & \VV 
\end{align*}
characterized by the moments
\[ \mu^{(a,1)}(x^k y^r) = (1-a^{k+r+1}) \Phi_{k,r} \]
where 
\[ \Phi_{k,r} = \begin{cases} 2G_{k+r+1} & \textrm{ if } k=0 \textrm{ or } r=0  \\
2 \sum_{n=1}^{\infty} q^n \sum_{dd'=n} d^{k}d'^{r} & \textrm{ if } k,r\neq 0. \end{cases} \]
Equivalently, 
\[ \Phi_{k,r} = \begin{cases} \theta^{r} 2G_{k+1 - r} & \textrm { if } k \geq r \\  
\theta^{k} 2G_{r+1 - k} & \textrm { if } r \geq k. \end{cases}\]

\begin{remark}
The symmetry between $k \geq r$ and $r \leq k$ becomes more complicated as soon as we consider locally constant functions as in Remark \ref{remark:interpolation-locally-constant}. 
\end{remark}

\subsubsection{Halving the measure}
Our technique will only recover ``half" of the measure, i.e. only the moments for $k \geq r$. To make this precise, consider the map 
\begin{align*}
\varphi: \ZZ_p \times \ZZ_p & \rightarrow \ZZ_p \times \ZZ_p \\
(x, y) & \mapsto (x, xy)  
\end{align*}
with image the subset of $(x,y)$ with $|y| \leq |x|$. The measure $\varphi_* \mu^{(a,1)}$ is characterized by the moments
\begin{equation}\label{eqn:halved-moments} \varphi_*\mu^{(a,1)} (x^s y^t) = \mu^{(a,1)}(x^{s+t}y^t) = (1-a^{s+2t+1})\theta^{t} 2G_{s+1} 
\end{equation}

\subsection{Convolution of the one-variable measure and the action map}

\begin{theorem} There is a $\VV$-valued measure $\nu$ on $\ZZ_p \times \ZZ_p$ with moments
\[ \nu (x^s y^t) = \psi( (1-a^{s+1})2 G_{s+1} ) (y^t) = (1-a^{s+1}) \theta^t 2G_{s+1}. \]
\end{theorem}
\begin{proof}
From the one-variable Eisenstein measure $\mu^{(a)}$ and the action map 
\[ \psi: \VV \rightarrow \mathrm{Meas} (\ZZ_p, \VV), \]
we obtain a bilinear form on $\Cont(\ZZ_p, \ZZ_p) \times \Cont(\ZZ_p, \ZZ_p)$ 
\[ (f, g) \mapsto \psi( \mu^{(a)}(f) )(g) \]
and thus, by Proposition \ref{prop:bilinear-form-to-measure}, a measure $\nu$ such that
\[ \nu (x^s y^t) = \psi( (1-a^{s+1})2 G_{s+1} ) (y^t) = (1-a^{s+1}) \theta^t 2G_{s+1}. \]
\end{proof}

Comparing with (\ref{eqn:halved-moments}), we see that the measure $\nu$ interpolates the same Eisenstein series as $\varphi_* \mu^{(a,1)}$, although with a different normalizing factor (recall that this normalizing factor removes the powers of $p$ in the denominator of the constant term $\zeta(1-k)$ of $G_k$ when $k\equiv - 1 \mod p-1$). We note\footnote{We thank an anonymous referee for pointing this out!} that this is enough to essentially recover the two variable $p$-adic $L$-function $\calL$ of \cite[Chapter VII]{katz:p-adic-interpolation-of-real-analytic-eisenstein-series}, because the map $\varphi$ induces an automorphism of $\l(\bbZ_p^\times\r)^2$ and the construction of $\calL$ passes through the restriction from $\bbZ_p^2$ to $\l(\bbZ_p^\times\r)^2$. Concretely, if $\chi$ is a continuous character of $(\bbZ_p^\times)^2$, $\chi(x,y)=\chi_1(x)\chi_2(y)$, then  
\[  \l(1- \frac{\chi_1(a) a}{\chi_2(a)}\r) \calL(\chi) =\nu \l( \frac{\chi_1(x)}{\chi_2(x)}\chi_2(y) \r), \]
where on the right we have implicitly extended the function $\frac{\chi_1(x)}{\chi_2(x)}\chi_2(y)$ on $(\bbZ_p^\times)^2$ by zero to obtain a function on $\bbZ_p^2$. 

 This type of construction may be useful for studying special values of families of automorphic forms and their images under differential operators on other Shimura varieties where explicit computations with $q$-expansions are not always available.

\bibliographystyle{plain}
\bibliography{refs}

\end{document}